\documentclass[12pt]{article}

\usepackage[margin=1.3in, top=1.3in, bottom=1.3in]{geometry}

\usepackage[utf8]{inputenc} 
\usepackage[T1]{fontenc}    
\usepackage{url}            
\usepackage{booktabs}       
\usepackage{amsfonts}       
\usepackage{nicefrac}       
\usepackage{microtype}      

\usepackage{amsmath,amssymb,amsthm,xcolor}
\usepackage{array}
\usepackage{float}

\newtheorem{remark}{Remark}
\newtheorem{theorem}{Theorem}
\newtheorem{lemma}{Lemma}
\newtheorem{corollary}{Corollary}

\newtheorem{proposition}{Proposition}
\newtheorem{definition}{Definition}
\usepackage{graphicx}
\usepackage{subcaption}
\usepackage{mdframed}
\usepackage{lipsum}
\usepackage{wrapfig}
\usepackage[hidelinks,hypertexnames=false]{hyperref}
\usepackage{tikz}
\usepackage{enumitem}
\allowdisplaybreaks
\usepackage{pgfplots}
\usetikzlibrary{intersections, pgfplots.fillbetween}
\usepackage{epstopdf}
\usepackage{algorithmic}
\usepackage{booktabs}
\usepackage{changepage}
\usepackage{multirow}
\usepackage{algorithm}

\title{\LARGE Global stability of first-order methods for coercive tame functions}

\begin{document}

\author{\large C\'edric Josz\thanks{\url{cj2638@columbia.edu}, IEOR, Columbia University, New York. Research supported by NSF EPCN grant 2023032 and ONR grant N00014-21-1-2282.} \and Lexiao Lai\thanks{\url{ll3352@columbia.edu}, IEOR, Columbia University, New York.}}
\date{}

\maketitle

\begin{center}
    \textbf{Abstract}
    \end{center}
    \vspace*{-3mm}
 \begin{adjustwidth}{0.2in}{0.2in}
~~~~We consider first-order methods with constant step size for minimizing locally Lipschitz coercive functions that are tame in an o-minimal structure on the real field. We prove that if the method is approximated by subgradient trajectories, then the iterates eventually remain in a neighborhood of a connected component of the set of critical points. Under suitable method-dependent regularity assumptions, this result applies to the subgradient method with momentum, the stochastic subgradient method with random reshuffling and momentum, and the random-permutations cyclic coordinate descent method.
\end{adjustwidth} 
\vspace*{3mm}

\section{Introduction}
\label{sec:Introduction}
Consider the unconstrained minimization problem
\begin{equation}
\label{eq:obj}
	\inf\limits_{x\in \mathbb{R}^n} f(x) := \frac{1}{N}\sum\limits_{i= 1}^N f_i(x)
\end{equation}
where $f_i:\mathbb{R}^n \rightarrow \mathbb{R}$ is locally Lipschitz for $i = 1,2, \ldots, N$. Such unconstrained optimization problems are central in machine learning applications such as empirical risk minimization \cite{bottou2018optimization}, low-rank matrix recovery \cite{li2019,ma2022global,zhang2022accelerating}, and the training of deep neural networks \cite{lecun2015deep}. We study some widely used first-order methods, namely the subgradient method with momentum (Algorithm \ref{alg:mag}), the stochastic subgradient method with random reshuffling and momentum (Algorithm \ref{alg:rrm}), and the random-permutations cyclic coordinate descent method (Algorithm \ref{alg:rcd}). While they are implemented by machine learning practitioners \cite{sutskever2013importance,tensorflow2015-whitepaper,paszke2019pytorch}, the analysis of these methods with constant step sizes seems to be absent from the literature when the objective is neither convex nor differentiable with a locally Lipschitz gradient (see Section \ref{sec:Literature review}).

In this paper, we provide global stability guarantees for first-order methods with constant step size for objective functions that are locally Lipschitz, coercive, and tame in an o-minimal structure on the real field (Definition \ref{def:o-minimal}). In order to do so, we show that the function values and the iterates of an iterative method eventually stabilize around some critical value (Theorem \ref{thm:converge}) and the set of critical points (Corollary \ref{cor:converge_critical}) respectively, given that the method is approximated by subgradient trajectories of the objective function (Definition \ref{def:approx_flow_new}). As it turns out, all of the aforementioned first-order methods are approximated by subgradient trajectories of locally Lipschitz functions under method-dependent regularity assumptions (Propositions \ref{prop:rrm_flow} and \ref{prop:cd_flow}) as summarized in Table \ref{tab:assumption} (random reshuffling with momentum is short for stochastic subgradient method with random reshuffling and momentum). Therefore, these methods fit into our framework and their stability is guaranteed by Theorem \ref{thm:converge} and Corollary \ref{cor:converge_critical}. To the best of our knowledge, these methods have not been studied before at such generality as in this paper. In particular, we do not require the objective function to be convex and we do not require it to be differentiable with a locally Lipschitz gradient.

The function class studied in this paper is well-suited for applications. Indeed, seemingly all continuous objective functions of interest nowadays are locally Lipshitz and tame in an o-minimal structure on the real field. Many objective functions arising in data science are coercive due to the use of regularizers. Some objectives are naturally coercive, such as in symmetric low-rank matrix recovery problems \cite{ge2016,li2019}. For functions that are not coercive, our results can still be applied if the iterates are uniformly bounded. We discuss this extension in Remark \ref{remark:coercive}.

\begin{table}
\caption{Standing assumption: $f$ is coercive and tame.}
\vspace*{-3mm}
\begin{center}
\begin{tabular}{|c|c|c|c|} 
\hline
Algorithm & Our assumption & Literature & Conclusion \\
\hline
\multirow{4}{2.7cm}{Subgradient method with momentum} & \multirow{4}{3cm}{$f$ locally Lipschitz} & 
\multirow{4}{3cm}{$f$ differentiable with locally Lipschitz gradient \cite{zavriev1993heavy,ochs2014ipiano}} & 
\multirow{9}{3.5cm}{$f(x_k)$ and $x_k$ eventually stay arbitrarily close to a critical value and a connected component of the set of critical points respectively, for all initial points in a bounded set $X_0$ and sufficiently small step sizes $\alpha$} 
\\ 
&  &  &\\ 
&  &  &\\ 
&  &  &\\ 
\cline{1-3}
\multirow{4}{3cm}{Random reshuffling with momentum} & \multirow{4}{3cm}{$f_i$ locally Lipschitz and subdifferentially regular} & \multirow{4}{2.5cm}{no results} & \\ 
& & &\\ 
& & &\\ 
& & &\\ 
\cline{1-3}
\multirow{4}{3cm}{Random-permutations cyclic coordinate descent method} & \multirow{4}{3cm}{$f$ continuously differentiable} & \multirow{4}{2.7cm}{$f$ differentiable with locally Lipschitz gradient \cite{beck2013convergence}} & \\ 
& & &\\ 
& & &\\ 
& & &\\
\hline
\end{tabular}
\end{center}
\label{tab:assumption}
\end{table}

Our results rely on the connection between the iterates of first-order methods and the subgradient trajectories of the objective function. The subgradient trajectories of a locally Lipschitz function are solutions to a differential inclusion (i.e., equation \eqref{eq:ivp} with $c = 1$). Previous works used the theory of differential inclusions \cite{aubin1984differential} to study the stochastic subgradient method. Most of them are in the setting where a stochastic subgradient oracle is available (i.e., it generates a subgradient of the objective function in expectation) and the step sizes are diminishing. It was shown that the iterates of stochastic subgradient method converge almost surely to an internally chain transitive set of a differential inclusion \cite[Theorem 3.6]{benaim2005stochastic} \cite[Corollary 4]{borkar2009stochastic}, with the proviso that the iterates are bounded almost surely and that the step sizes are not summable but square summable, among other assumptions. By additionally assuming that the objective function is Whitney stratifiable, the iterates subsequentially converge to critical points and the function values converge to a critical value almost surely \cite[Corollary 5.9]{davis2020stochastic}. 

In contrast to the above works, we are interested in random reshuffling or permutation, which includes sampling without replacement, and constant step sizes. The recent work of Bianchi \textit{et al.} \cite{bianchi2022convergence} uses differential inclusions to analyze the stochastic subgradient method with constant step size. When a stochastic subgradient oracle is available along with other assumptions involving a Markov kernel, the iterates of the stochastic subgradient method eventually lie in the neighborhood of the critical points with high probability \cite[Theorem 3]{bianchi2022convergence} for sufficiently small step sizes. The stochastic subgradient method with random reshuffling and diminishing step sizes was analyzed via differential inclusions recently in the work of Pauwels \cite{pauwels2021incremental}. For Lipschitz continuous objectives, bounded iterates converge subsequentially to the set of critical points with respect to a conservative field \cite[Corollary 6]{pauwels2021incremental}.

We next give the update rules of the first-order methods considered in this paper. Algorithm \ref{alg:mag} is a generalization the framework proposed in the work of Kovachki and Stuart \cite[(7)]{kovachki2021continuous} from differentiable functions to locally Lipschitz functions. We denote by $\partial f$ the Clarke subdifferential \cite{clarke1990} (see Definition \ref{def:Clarke}) of a locally Lipschitz function $f$. Algorithm \ref{alg:mag} reduces to the heavy ball method \cite{polyak1964some} when $\gamma = 0$ and to the Nesterov's accelerated subgradient method \cite[equation (2.2.22)]{nesterov2018introductory} when $\beta = \gamma$ respectively. It also includes the vanilla subgradient method as a special case when $\beta = \gamma= 0$. Algorithm \ref{alg:rrm} is an extension of Algorithm \ref{alg:mag} which exploits the composite nature of the objective function \eqref{eq:obj}. Its update is the same as Algorithm \ref{alg:mag} except that each step concerns only one component $f_i$, which is chosen at a random order at every iteration (epoch). This is exactly how stochastic subgradient method with momentum is implemented in practice (see for e.g., documentations from TensorFlow\footnote{\url{https://www.tensorflow.org/api\_docs/python/tf/keras/optimizers/SGD}}, PyTorch\footnote{\url{https://pytorch.org/docs/stable/generated/torch.optim.SGD.html}} and scikit-learn\footnote{\url{https://scikit-learn.org/stable/modules/sgd.html}}). Last, Algorithm \ref{alg:rcd} is the random-permutations cyclic coordinate descent method, where $\nabla_i f(x):= [\nabla f(x)]_i e_i$, $[\nabla f(x)]_i$ is the $i$\textsuperscript{th} entry of $\nabla f(x)$, and $e_i$ is the $i$\textsuperscript{th} vector in the canonical basis of $\mathbb{R}^n$. Similar to Algorithm \ref{alg:rrm}, Algorithm \ref{alg:rcd} chooses a permutation of all the coordinates at every iteration and cycles through them.

The paper is organized as follows. Section \ref{sec:Literature review} contains a literature review on the first-order methods and stochastic approximations with constant step size. Section \ref{sec:global convergence} contains the global stability results for iterative methods that are approximated by subgradient trajectories. Finally, Section \ref{sec:Continuous-time} explains how the first-order methods fit into the abstract framework of Section \ref{sec:global convergence}.
\begin{algorithm}[ht]
\caption{Subgradient method with momentum}\label{alg:mag}
\begin{algorithmic}
\STATE{\textbf{choose} step size $\alpha>0$, momentum parameters $\beta \in (-1,1)$, $\gamma \in \mathbb{R}$, constant $\delta>0$, $x_{-1},x_0 \in \mathbb{R}^n$ with $\|x_{-1}-x_0\|\leqslant \delta \alpha$}
\FOR{$k = 0,1, \ldots$}
    \STATE{ $y_k = x_k + \gamma (x_k - x_{k-1})$}
    \STATE{ $x_{k+1}  \in x_k + \beta (x_k - x_{k-1})  - \alpha \partial f (y_k)$ } 
\ENDFOR
\end{algorithmic}
\end{algorithm}

\begin{algorithm}[!ht]
\caption{Random reshuffling with momentum}\label{alg:rrm}
\begin{algorithmic}
\STATE{\textbf{choose} step size $\alpha>0$, momentum parameters $\beta \in (-1,1)$, $\gamma \in \mathbb{R}$, constant $\delta>0$, $x_{-1,N-1}, x_0 \in \mathbb{R}^n$ with $\|x_{-1,N-1}- x_0\| \leqslant \delta \alpha$}
\FOR{$k = 0,1, \ldots$}
    \STATE{ $x_{k,0} = x_k$}
    \STATE{$x_{k,-1} = x_{k-1,N-1}$}
    \STATE{choose a permutation $\sigma^k$ of $\{1,2,\ldots, N\}$}
    \FOR{$i = 1,2, \ldots, N$}
        \STATE{ $y_{k,i} = x_{k,i-1} + \gamma (x_{k,i-1} - x_{k,i-2})$}
        \STATE{ $x_{k,i} \in x_{k,i-1} + \beta (x_{k,i-1} - x_{k,i-2})  - \alpha \partial f_{\sigma^k_i} (y_{k,i})$}
    \ENDFOR
    \STATE{ $x_{k+1} = x_{k,N}$}
\ENDFOR
\end{algorithmic}
\end{algorithm}

\begin{algorithm}[ht]
\caption{Random-permutations cyclic coordinate descent method}\label{alg:rcd}
\begin{algorithmic}
\STATE{\textbf{choose}  $x_0 \in \mathbb{R}^n$, step size $\alpha>0$}
\FOR{$k = 0,1, \ldots$}
    \STATE{choose a permutation $\sigma^k$ of $\{1,2,\ldots, n\}$}
    \STATE{ $x_{k,0} = x_k$}
    \FOR{$i = 1,2, \ldots, n$}
        \STATE{ $x_{k,i} = x_{k,i-1} - \alpha \nabla_{\sigma^k_i} f(x_{k,i-1})$}
    \ENDFOR
    \STATE{ $x_{k+1} = x_{k,n}$}
\ENDFOR
\end{algorithmic}
\end{algorithm}

\section{Literature review}
\label{sec:Literature review}

The gradient method with momentum with $\gamma = 0$ was introduced by Polyak \cite{polyak1964some}. It admits a nearly optimal local convergence rate for twice continuously differentiable strongly convex functions \cite[Theorem 9]{polyak1964some}. Nesterov showed that it admits a globally optimal convergence rate \cite[Theorem 2.1.13]{nesterov2018introductory} if one chooses $\beta = \gamma$ in an appropriate manner. With variable momentum parameters, it also has an optimal rate for convex functions with Lipschitz gradients whose infimum is attained \cite{nesterov1983method}. If one relaxes the convexity assumption, then with a suitable choice of parameters $\alpha,\beta,$ and $\gamma$, the gradients $\nabla f(x_k)$ converge to zero \cite[Lemmas 1,2,3]{zavriev1993heavy} for any initial points $x_{-1},x_0 \in \mathbb{R}^n$. If in addition $f$ is coercive and satisfies the Kurdyka-\L{}ojasiewicz inequality \cite{kurdyka1998gradients} at every point and $x_{-1}=x_0$, then the iterates have finite length \cite[Theorem 4.9]{ochs2014ipiano}. In the nonsmooth setting that we consider in this paper, there seems to be no results to the best of our knowledge. 

The incremental subgradient method is a special of the stochastic subgradient method with random reshuffling where the components are visited in a fixed order. It can be traced back to the Widrow-Hoff least mean squares method \cite{widrow1960adaptive} for minimizing a finite sum of convex quadratics in 1960. It was pointed out later by Kohonen that with sufficiently small constant step sizes, the limit points of the iterates of the least mean squares method are close to a minimum of the objective function \cite{kohonen1974adaptive}. With diminishing step sizes that are not summable but square summable, the least mean squares method converges to a minimum of the problem \cite{luo1991convergence}. For convex objectives, the incremental subgradient method with constant step size $\alpha>0$ satisfies $\liminf_{k \rightarrow \infty} f(x_k) \leqslant \inf_{ \mathbb{R}^n} f + C\alpha$ for some $C>0$ \cite[Proposition 2.1]{nedic2001incremental}, provided that $\inf_{\mathbb{R}^n} f>-\infty$ and the subgradients of the components $f_i$ are uniformly bounded. We refer the readers to the survey paper \cite{bertsekas2011incremental}, the textbook \cite{bertsekas2015convex}, and references therein for a more detailed discussion on the subject.

The stochastic subgradient method with random reshuffling is a stochastic version of the incremental subgradient method. It was shown recently that the stochastic gradient method with random reshuffling outperforms the incremental gradient method in expectation on strongly convex functions with quadratic components \cite[Theorem 2]{gurbuzbalaban2021random}, under certain choices of diminishing step sizes. If the objective function is strongly convex and differentiable with Lipschitz gradients among other assumptions, then the iterates and the corresponding function values of the stochastic gradient method with random reshuffling and constant step size eventually lie in a neighborhood of the minimizer \cite[Theorem 1]{mishchenko2020random} and a neighborhood of the minimum \cite[Theorem 1]{nguyen2021unified} respectively, both in expectation. By relaxing the strong convexity assumption to mere convexity, the function values evaluated at the average iterates $\hat{x}_k := (\sum_{l = 0}^k x_l)/k$ eventually lie in a neighborhood of the minimum in expectation \cite[Theorem 3]{mishchenko2020random} \cite[Remark 1]{nguyen2021unified}. By further removing the convexity assumption, the minimum norm of the gradients eventually lies in a neighborhood of zero in expectation \cite[Theorem 4]{mishchenko2020random} \cite[Corollary 1, Corollary 3]{pauwels2021incremental} (see also \cite[Theorem 4]{nguyen2021unified} for a similar result). The long-term behavior of the iterates for nonconvex and nonsmooth objective functions has so far remained elusive.

Despite the empirical success of incorporating momentum into the incremental gradient method/stochastic gradient method with random reshuffling \cite{sutskever2013importance}, the theoretical understanding of such methods is limited. So far, the only guarantees available are for modified versions \cite{tran2021smg,tran2022nesterov}. The work of Tran \textit{et al.} \cite{tran2021smg} in 2021 studied a modified version of stochastic gradient method with random reshuffling and heavy ball. The momentum is constant within every iteration (epoch) and is equal to the average of the gradients evaluated in the previous epoch. With the modification, the norm of gradients of the average iterates $\hat{x}_k$ eventually lie in a neighborhood of zero in expectation \cite[Corollary 1]{tran2021smg}, under various assumptions \cite[Assumption 1]{tran2021smg}. A modified stochastic gradient method with random reshuffling and Nesterov's momentum was studied recently \cite{tran2022nesterov}. The momentum is only applied at the level of the outer loop, at the end of each iteration (epoch). In this setting, the function values eventually lie a neighborhood of the minimum when the component functions are convex \cite[Theorem 1]{tran2022nesterov}, among other assumptions.

Coordinate descent methods are the object of the survey paper \cite{wright2015coordinate} by Wright in 2015. The idea of coordinate descent methods is to optimize with respect to one variable at a time. It was first studied under the framework of univariate relaxation \cite[Section 14.6]{ortega1970}. With exact line search and almost cyclic rule or Gauss-Southwell rule for cycling over the coordinates, the coordinate descent method converges linearly to a minimizer of a strongly convex objective that is twice differentiable \cite[Theorem 2.1]{luo1992convergence}. More recently, global convergence of random coordinate descent method was established for convex objectives with Lipschitz continuous partial derivatives \cite{nesterov2012efficiency}. In contrast to cyclic coordinate descent methods, random coordinate descent methods choose a coordinate randomly at each iteration instead of following a cycling rule. Similar to the stochastic subgradient method with random reshuffling, the random-permutations cyclic coordinate descent method considered in this work is easier to implement than the random coordinate descent method as it requires only sequential access of the data \cite{gurbuzbalaban2020randomness}. Using \cite[Lemma 3.3, remark 3.2]{beck2013convergence}, the convergence of the random-permutations cyclic coordinate descent method can be deduced for coercive functions with locally Lipschitz gradients. The superior performance of the random-permutations cyclic coordinate descent method was observed in numerical experiments, and was supported by analysis for convex quadratic objectives \cite{lee2019random,gurbuzbalaban2020randomness}. For objective functions without a locally Lipschitz gradient, the study of the method appears to be absent from the literature.

Stochastic approximations of differential inclusions with constant step size have led to recent advances on an oracle-based stochastic subgradient method \cite{bianchi2022convergence}. Given differential equations with Lipschitz right-hand sides over a finite time horizon, Kurtz proposed a sequence of discrete time stochastic processes that approaches their solutions with a probability that goes to one \cite[Theorem (4.7)]{kurtz1970solutions} (see also \cite[Proposition 3.1]{benaim1998recursive}). Over an infinite time horizon, the sequence of corresponding invariant measures concentrates around the Birkhoff center of the differential equations \cite[Corollary 3.2]{benaim1998recursive}. Later, Roth and Sandholm \cite{roth2013stochastic} extended these results to differential inclusions with upper semicontinuous right-hand sides and compact supports, along with other assumptions. More recently, the work of Bianchi \emph{et al.} \cite{bianchi2019constant} studied stochastic approximation with constant step size under a different set of assumptions, relaxing the compact support assumption from the previous literature. We refer the readers to the textbooks on Markov processes and stochastic approximations \cite{ethier2009markov,benveniste2012adaptive} for more references on the subject. Although the above results cannot be directly 
 applied to the settings of this work, readers will see that we adopt similar proof strategies when studying the relationship between the discrete and continuous dynamics. For example, we also study the subsequential convergence of the linear interpolation of the iterates to the solutions of a continuous-time system. More discussion on this matter is deferred to the following section in Remark \ref{remark:dontchev}. In addition, our analysis relies on the theory of set-valued analysis \cite{clarke1990} and differential inclusion \cite{aubin1984differential}, which was also used in the aforementioned literature.

\section{Global stability of first-order methods}
\label{sec:global convergence}
We refer to an iterative method with constant step size as a set-valued mapping $\mathcal{M}: \mathbb{R}^{(\mathbb{R}^n)} \times (0,\infty) \times 2^{(\mathbb{R}^n)} \times \mathbb{N}\rightrightarrows (\mathbb{R}^n)^\mathbb{N}$ which, to an objective function $f:\mathbb{R}^n \rightarrow \mathbb{R}$, a constant step size $\alpha \in (0,\infty)$, a set $X_0\subset \mathbb{R}^n$, and a natural number $\bar{k}$ associates a set of sequences in $\mathbb{R}^n$ whose $\bar{k}$\textsuperscript{th} term is contained in $X_0$.

We next introduce several definitions. Let $\|\cdot\|$ be the induced norm of an inner product $\langle \cdot, \cdot\rangle$ on $\mathbb{R}^n$. Given a subset $S$ of $\mathbb{R}^n$ and $x\in \mathbb{R}^n$, consider the distance of $x$ to $S$ defined by $d(x,S) := \inf \{ \|x-y\| : y \in S \}$. Let $B(a,r)$ denote the closed ball of center $a\in \mathbb{R}^n$ and radius $r > 0$, and let $B(S,r):= \cup_{a \in S} B(a,r)$ where $S \subset \mathbb{R}^n$. Recall that a function $f:\mathbb{R}^n\rightarrow\mathbb{R}^m$ is locally Lipschitz if for all $a \in \mathbb{R}^n$, there exist $r>0$ and $L>0$ such that 
$\|f(x)-f(y)\| \leqslant L \|x-y\|$ for all $x,y \in B(a,r)$. We use $[f \leqslant \Delta]:= \{x\in \mathbb{R}^n: f(x) \leqslant \Delta\}$ to denote a sublevel set of a function $f:\mathbb{R}^n \rightarrow \mathbb{R}$ where $\Delta \in \mathbb{R}$.
A function $f:\mathbb{R}^n\rightarrow\mathbb{R}$ is coercive if $\lim_{\|x\|\rightarrow \infty} f(x) =\infty$.

\begin{definition}{\cite[Chapter 2]{clarke1990}}
    \label{def:Clarke}
    Let $f:\mathbb{R}^n \rightarrow \mathbb{R}$ be a locally Lipschitz function. The Clarke subdifferential is the set-valued mapping $\partial f:\mathbb{R}^n\rightrightarrows\mathbb{R}^n$ defined for all $x \in \mathbb{R}^n$ by $\partial f(x) := \{ s \in \mathbb{R}^n : f^\circ(x,d) \geqslant \langle s , d \rangle, \forall d\in \mathbb{R}^n \}$ where
\begin{equation*}
    f^\circ(x,d) := \limsup_{\tiny\begin{array}{c} y\rightarrow x \\
    t \searrow 0
    \end{array}
    } \frac{f(y+td)-f(y)}{t}.
\end{equation*}
\end{definition}

We say that $x \in \mathbb{R}^n$ is critical if $0 \in \partial f(x)$, and that $v \in \mathbb{R}$ is a critical value if there exists $x\in \mathbb{R}^n$ such that $0 \in \partial f(x)$ and $v = f(x)$. If $f$ is continuously differentiable, then $\partial f(x) = \{\nabla f(x)\}$ \cite[2.2.4 Proposition]{clarke1990}.

\begin{definition}
    \label{def:absolutely_continuous} 
    \cite[Definition 1 p. 12]{aubin1984differential}
    Given some real numbers $a$ and $b$ such that $a<b$, a function $x(\cdot)$ defined from $[a,b]$ to $\mathbb{R}^n$ is absolutely continuous if for all $\epsilon>0$, there exists $\delta>0$ such that, for any finite collection of disjoint subintervals $[a_1,b_1],\hdots,[a_m,b_m]$ of $[a,b]$ such that $\sum_{i=1}^m b_i-a_i \leqslant \delta$, we have $\sum_{i=1}^m \|x(b_i) - x(a_i)\| \leqslant \epsilon$.
\end{definition}
By virtue of \cite[Theorem 20.8]{nielsen1997introduction}, a function $x:[a,b]\rightarrow\mathbb{R}^n$ is absolutely continuous if and only if it is differentiable almost everywhere on $(a,b)$, its derivative $x'(\cdot)$ is Lebesgue integrable, and $\forall t\in [a,b], x(t) - x(a) = \int_a^t x'(t)dt$.

The next definition is inspired by a series of works that resort to continuous-time dynamics to analyze discrete-time dynamics. The idea that discrete dynamics resemble their continuous counterpart dates back to Euler \cite{euler1792institutiones,blanton2006foundations}. He proposed discretizing ordinary differential equations to find approximate solutions. This technique is also used to prove the existence of solutions via the Cauchy Peano theorem \cite[Theorem 1.2]{coddington1955theory}. Ljung \cite{ljung1977analysis} and Kushner \cite{kushner1977general,kushner1977convergence} established a connection between the asymptotic behavior of discrete and continuous dynamics with noise, which is particularly useful when they are governed by conservative fields. Benaïm \textit{et al.} \cite{benaim2005stochastic,benaim2006stochastic,benaim2006dynamics} strengthened this connection by relaxing some assumptions and incorporating set-valued dynamics. Due to its importance in analyzing optimization algorithms in recent years \cite{borkar2009stochastic,duchi2018stochastic,davis2020stochastic,bolte2020conservative,salim2018random,pauwels2021incremental}, we elaborate on their contribution.

Benaïm, Hofbauer, and Sorin consider a closed set-valued mapping $F:\mathbb{R}^n \rightrightarrows \mathbb{R}^n$ with nonempty convex compact values for which there exists $C>0$ such that $\sup \{ \|s\| : s \in F(x)\} \leqslant C(1+\|x\|)$ for all $x\in \mathbb{R}^n$. They show that discrete trajectories of $F$ can be approximated by its continuous trajectories in the following sense. Let $(x_k)_{k\in \mathbb{N}}$ be a bounded sequence such that $x_{k+1} \in x_k + \alpha_k F(x_k)$ for all $k\in\mathbb{N}$ where $\alpha_k>0$, $\sum_{k=0}^\infty \alpha_k = \infty$, and $\sum_{k=0}^\infty \alpha_k^2 < \infty$ (Ljung and Kushner also assume this, following Robbins and Monro \cite{robbins1951stochastic}). Let $t_0 := 0$ and $t_k := \alpha_0+\hdots+\alpha_{k-1}$ for $k\geqslant 1$. Consider the linear interpolation defined by 
\begin{equation*}
    x(t) := x_k + \frac{t-t_k}{t_{k+1}-t_k}(x_{k+1}-x_k), ~~~\forall t\in[t_k,t_{k+1})
\end{equation*}
as well as the time shifted interpolations $x^{\tau}(\cdot) := x(\tau + \cdot)$ where $\tau\geqslant 0$. The key insight is that for any sequence $\tau_k\rightarrow \infty$, the shifted interpolations $x^{\tau_k}$ subsequentially converge to a solution to the differential inclusion $x'(t) \in F(x(t))$ for almost every $t>0$ in the topology of uniform convergence on compact intervals \cite[Theorem 4.2]{benaim2005stochastic}. When one specifies that $F:= -\partial f$ where $f:\mathbb{R}^n\rightarrow\mathbb{R}$ is a locally Lipschitz function, this can used to derive asymptotic properties of some important algorithms in optimization.

In this work, we are interested in the constant step size regime for which it is hopeless to try to establish uniform convergence over compact intervals (for a fixed discrete trajectory, as above). We thus ask for something weaker from the algorithms we analyze, namely, that the continuous and discrete time dynamics are close in uniform norm up to a certain time. We ask that this holds for a set of time shifted trajectories to account for multistep methods. We next give a precise meaning to this notion. We will use $\lfloor t \rfloor$ to denote the floor of a real number $t$ which is the unique integer such that $\lfloor t \rfloor \leqslant t < \lfloor t \rfloor + 1$.

\begin{definition}
\label{def:approx_flow_new}
An iterative method $\mathcal{M}$ is approximated by subgradient trajectories of a locally Lipschitz function $f:\mathbb{R}^n \rightarrow \mathbb{R}$ (up to a positive multiplicative constant) if there exists $c>0$ such that for any compact sets $X_0, X_1 \subset \mathbb{R}^n$, there exists $T>0$ such that for all $\epsilon>0$, there exists $\bar{\alpha}>0$ such that for all $\alpha \in (0,\bar{\alpha}]$, $\bar{k} \in \mathbb{N}$, and $(x_k)_{k \in \mathbb{N}} \in \mathcal{M}(f,\alpha,X_0,\bar{k})$ for which $x_0, \ldots, x_{\bar{k}} \in X_1$, there exists an absolutely continuous function $x:[0,T]\rightarrow \mathbb{R}^n$ such that
\begin{equation}
    \label{eq:ivp}
    x'(t) \in -c\partial f(x(t)),~~~\text{for almost every}~t\in [0,T],~~~x(0)\in X_0,
\end{equation}
and $\|x_{k}-x((k-\bar{k})\alpha)\|\leqslant \epsilon$ for $k=\bar{k},\hdots,\bar{k}+\lfloor T/\alpha\rfloor$.
\end{definition}

\begin{remark}
\label{remark:dontchev}
In the next section, we show that Algorithms \ref{alg:mag}, \ref{alg:rrm}, and \ref{alg:rcd} satisfy Definition \ref{def:approx_flow_new} (see Propositions \ref{prop:rrm_flow} and \ref{prop:cd_flow}). In order to do so, we always use the same strategy which consists in taking sequences generated by a given method with smaller and smaller constant step size, and show that a subsequence of their linear interpolations converges uniformly to a subgradient trajectory up to a finite time.

Several discretization methods of initial value problems with differential inclusions were studied in \cite{taubert1981converging,aubin1984differential,clarke1990,filippov2013differential} (see also a survey on the subject by Dontchev and Lempio \cite{dontchev1992difference}). Assume that the set-valued mapping underlying the differential inclusion is upper semicontinuous with nonempty compact convex values, such that the norm of their elements are upper bounded by a linear function of the norm of the argument. Then over any finite time horizon, a subsequence of linear interpolations of the Euler method with smaller and smaller step sizes converges uniformly to a solution to the initial value problem \cite[Theorem 2.2]{dontchev1992difference}. If in addition the set-valued mapping is bounded, then a class of linear multistep methods has the same convergence property as above \cite[p. 127, Theorem]{taubert1981converging} (see also \cite[Convergence Theorem 3.2]{dontchev1992difference}). We build on the techniques developed in the above works when checking Definition \ref{def:approx_flow_new}. We adapt them so that they can handle the case where $\partial f$
is not accessible (as in Algorithms \ref{alg:rrm} and \ref{alg:rcd}) and the set of initial points is a compact set. 
\end{remark}

The class of locally Lipschitz functions is too broad to obtain any meaningful results on the first-order methods \cite{rios2022examples,daniilidis2020pathological}. We thus consider functions that are tame in o-minimal structures. O-minimal structures (short for order-minimal) were originally considered by van den Dries, Pillay and Steinhorn \cite{van1984remarks,pillay1986definable}. They are founded on the observation that many properties of semi-algebraic sets can be deduced from a few simple axioms \cite{van1998tame}. Recall that a subset $A$ of $\mathbb{R}^n$ is semi-algebraic \cite{bochnak2013real} if it is a finite union of basic semi-algebraic sets, which are of the form $\{ x \in \mathbb{R}^n : p_i(x) = 0, ~ i = 1,\hdots,k; ~  p_i(x) > 0, ~ i = k+1,\hdots,m \}$ where $p_1,\hdots,p_m \in \mathbb{R}[X_1,\hdots,X_n]$ (i.e., polynomials with real coefficients). 

\begin{definition}\cite[Definition p. 503-506]{van1996geometric}
\label{def:o-minimal}
An o-minimal structure on the real field is a sequence $S = (S_k)_{k \in \mathbb{N}}$ such that for all $k \in \mathbb{N}$:
\begin{enumerate}
    \item $S_k$ is a boolean algebra of subsets of $\mathbb{R}^k$, with $\mathbb{R}^k \in S_k$;
    \item $S_k$ contains the diagonal $\{(x_1,\hdots,x_k) \in \mathbb{R}^k : x_i = x_j\}$ for $1\leqslant i<j \leqslant k$;
\item If $A\in S_k$, then $A\times \mathbb{R}$ and $\mathbb{R}\times A$ belong to $S_{k+1}$;
    \item If $A \in S_{k+1}$ and $\pi:\mathbb{R}^{k+1}\rightarrow\mathbb{R}^k$ is the projection onto the first $k$ coordinates, then $\pi(A) \in S_k$;
    \item $S_3$ contains the graphs of addition and multiplication;
    \item $S_1$ consists exactly of the finite unions of open intervals and singletons. 
\end{enumerate}
\end{definition}

Note that $S_1$ are the semi-algebraic subsets of $\mathbb{R}$ and by \cite[2.5 Examples (3)]{van1996geometric}, $S_k$ contains the semi-algebraic subsets of $\mathbb{R}^k$. A subset $A$ of $\mathbb{R}^n$ is definable in an o-minimal structure $(S_k)_{k\in\mathbb{N}}$ if $A \in S_n$. A function $f:\mathbb{R}^n\rightarrow\mathbb{R}$ is definable in an o-minimal structure if its graph, that is to say $\{(x,t) \in \mathbb{R}^{n+1} : f(x)=t \}$, is definable in that structure. A set $C \subset \mathbb{R}^n$ is tame \cite{ioffe2009} in an o-minimal structure $(S_k)_{k\in \mathbb{N}}$ if
    $$
    \forall x\in \mathbb{R}^{n}, ~ \forall r>0, ~~~ C \cap B(x,r) \in S_{n}.
    $$
    and a function $f:\mathbb{R}^n\rightarrow \mathbb{R}$ is tame if its graph is tame. With the above definitions, we are now ready to state two technical lemmas. The first relates a uniform neighborhood of a sublevel set with another sublevel set. The second is analogous to the descent lemma for smooth functions \cite[Lemma 1.2.3]{nesterov2018introductory} \cite[Lemma 5.7]{beck2017first}.

\begin{lemma}
\label{lemma:containment}
    Let $f:\mathbb{R}^n \rightarrow \mathbb{R}$ be a locally Lipschitz function. Let $\Delta \in \mathbb{R}$ and let $L>0$ be a Lipschitz constant of $f$ in $[f\leqslant \Delta]$. For any $\epsilon'>0$, $B([f \leqslant \Delta - \epsilon'L], \epsilon') \subset [f \leqslant \Delta]$.
\end{lemma}
\begin{proof}
We show that $B(a,\epsilon') \subset [f\leqslant \Delta]$ for all $a \in [f\leqslant \Delta - \epsilon' L]$. Indeed, if $b \in B(a,\epsilon') \setminus [f\leqslant \Delta]$, then there exists $c$ in the segment $[a,b)$ such that $f(c) = \Delta$ and $ \epsilon' L = \Delta -(\Delta - \epsilon' L) \leqslant f(c)-f(a) \leqslant L\|c-a\| < \epsilon'L$. 
\end{proof}

\begin{lemma}
\label{lemma:decrease}
Let $f:\mathbb{R}^n \rightarrow \mathbb{R}$ be a locally Lipschitz tame function and let $\mathcal{M}$ be an iterative method with constant step size. Let $X\subset \mathbb{R}^n$ and $L$ be a Lipschitz constant of $f$ on $X$. For all $T, \epsilon', \alpha, c > 0$, $\bar{k} \in \mathbb{N}$, $(x_k)_{k\in \mathbb{N}} \in (\mathbb{R}^n)^\mathbb{N}$, and for any subgradient trajectory $x:[0,T]\rightarrow \mathbb{R}^n$ of $cf$ such that $x([0,T]) \subset X$, $x_k \in X$, and $\|x_k - x(\alpha (k-\bar{k}))\|\leqslant \epsilon'$ for $k = \bar{k}, \ldots, \bar{k}+\lfloor T/\alpha\rfloor$, we have
	 \begin{equation*}
	     f(x_k) \leqslant f(x((k - \bar{k})\alpha)) + \epsilon' L \leqslant f(x_{\bar{k}})- c\int_0^{(k - \bar{k})\alpha} d(0,\partial f (x(s)))^2~ds + 2\epsilon' L
	 \end{equation*}
	 for $k = \bar{k},\ldots, \bar{k} + \lfloor T/\alpha\rfloor$.
\end{lemma}
\begin{proof}
    For $k=\bar{k},\hdots,\bar{k}+\lfloor T/\alpha\rfloor$, we have
	 \begin{subequations}
     \label{eq:fxk}
    	\begin{align}
		f(x_k) &\leqslant f(x((k - \bar{k})\alpha)) + \epsilon' L \label{eq:fxk_a}\\[2mm]
		&= f(x(0)) - (f(x(0)) - f(x((k - \bar{k})\alpha))) + \epsilon' L\label{eq:fxk_b}\\[2mm]
     &\leqslant f(x_0) - (f(x(0)) - f(x((k - \bar{k})\alpha))) + 2\epsilon' L\label{eq:fxk_c}\\[1mm]
		&= f(x_0)- c\int_0^{(k - \bar{k})\alpha} d(0,\partial f (x(s)))^2~ds + 2\epsilon' L. \label{eq:fxk_d}
    	\end{align}
     \end{subequations}
	 In \eqref{eq:fxk_a} and \eqref{eq:fxk_c}, we invoke the Lipschitz constant $L$ of $f$ on $X \ni x((k - \bar{k})\alpha),x_k$. \eqref{eq:fxk_d} is a consequence of \cite[Lemma 5.2, Theorem 5.8]{davis2020stochastic} (see also \cite{drusvyatskiy2015curves}).
\end{proof}

We now turn to our main results, namely Theorem \ref{thm:converge} and Corollary \ref{cor:converge_critical}, which we prove using Lemmas \ref{lemma:containment} and \ref{lemma:decrease}.

\begin{theorem}[Stability of function values]
\label{thm:converge}
	Let $f:\mathbb{R}^n \rightarrow \mathbb{R}$ be a locally Lipschitz coercive tame function and let $\mathcal{M}$ be an iterative method with constant step size that is approximated by subgradient trajectories of $f$. For any bounded set $X_0 \subset \mathbb{R}^n$ and $\epsilon>0$, there exist $\bar{\alpha}, \Delta>0$ such that for all $(x_k)_{k \in \mathbb{N}} \in \mathcal{M}(f,(0,\bar{\alpha}],X_0,0)$, we have $f(x_k) \leqslant \Delta$ for all $k \in \mathbb{N}$ and there exist a critical value $f^*$ of $f$ and $k_0 \in \mathbb{N}$ such that $ |f(x_k) - f^*| \leqslant \epsilon$ for all $k \geqslant k_0$.
\end{theorem}

\begin{proof}
Let $f:\mathbb{R}^n \rightarrow \mathbb{R}$ be a locally Lipschitz coercive tame function. Since $f$ is tame and coercive, there exists $\Delta>0$ such that $X_0 \subset [f\leqslant \Delta/2]$ and $\Delta$ is not a critical value of $f$. By the definable Morse-Sard theorem \cite[Corollary 9]{bolte2007clarke}, $f$ has finitely many critical values $f_1>\cdots>f_p$ in $[f\leqslant \Delta]$ (and it has at least one since $f$ is coercive and continuous). Since $f$ is coercive and continuous, the compact sublevel sets $[|f-f_i|\leqslant \epsilon]$, $i = 1,\hdots,p$, are pairwise disjoint after possibly reducing $\epsilon$, which we may do without loss of generality. We may also assume that $f_1+2\epsilon\leqslant\Delta$. According to the Kurdyka-\L{}ojasiewicz inequality \cite[Theorem 14]{bolte2007clarke} (see also \cite[Theorem 4.1]{attouch2010proximal}) and the monotonicity theorem \cite[(1.2) p. 43]{van1998tame} \cite[Lemma 2]{kurdyka1998gradients}, there exist $\rho>0$ and a strictly increasing concave continuous definable function $\psi:[0,\rho) \rightarrow [0,\infty)$ that is continuously differentiable on $(0,\rho)$ with $\psi(0) = 0$ such that $d(0,\partial f(x)) \geqslant 1/\psi'(|f(x) - f_i|)$ for all $x \in [|f-f_i|\leqslant \epsilon]$ whenever $0<|f(x)-f_i|< \rho$ for $i = 1, \ldots, p$. Without loss of generality, we assume $\epsilon<\rho$ so that $d(0,\partial f(x)) \geqslant 1/\psi'(|f(x) - f_i|)$ for all $x \in [|f-f_i|\leqslant \epsilon]$ such that $f(x) \neq f_i$.

Consider a Lipschitz constant $L\geqslant 1$ of $f$ in $[f\leqslant \Delta]$ and the quantity
	 \begin{equation}
	 \label{eq:M}
	     M := \inf\{d(0,\partial f(x)):|f(x) - f_i|\geqslant \epsilon/2, ~ i =1, \ldots, p,~ f(x)\leqslant \Delta\}>0.
	 \end{equation}
	 Since $\mathcal{M}$ is approximated by subgradient trajectories of $f$, by Definition \ref{def:approx_flow_new} there exist $c,T>0$, and $\bar{\alpha} \in (0,T/2)$ such that
	 such that for all $\alpha \in (0,\bar{\alpha}]$, $\bar{k}\in \mathbb{N}$, and $(x_k)_{k \in \mathbb{N}} \in \mathcal{M}(f,\alpha,[f\leqslant \Delta/2],\bar{k})$ for which $x_0, \ldots, x_{\bar{k}} \in [f \leqslant \Delta]$, there exists a subgradient trajectory $x:[0,T]\rightarrow \mathbb{R}^n$ of $f$ up to the multiplicative constant $c$ for which $x(0)\in [f\leqslant\Delta/2]$ and
	 $\|x_k - x(\alpha (k - \bar{k}))\|\leqslant \epsilon'$ for $k = \bar{k}, \ldots, \bar{k}+ \lfloor T/\alpha\rfloor$ where
	 \begin{equation*}
	     \epsilon' := \min\left\{\frac{\Delta}{4L}, \frac{cM^2T}{24L}, \frac{\epsilon}{8L}, \frac{cT}{2L\psi'(\epsilon/2)^2}\right\}>0.
	 \end{equation*}
	 Since $[|f-f_1|\leqslant \epsilon],\ldots,[|f-f_p|\leqslant \epsilon]$ are compact, after possibly reducing $T$ and $\bar{\alpha}$ the statement still holds if one replaces the initial set $[f\leqslant\Delta/2]$ by $[|f-f_1|\leqslant \epsilon],[|f-f_2|\leqslant \epsilon],\hdots,$ or $[|f-f_p|\leqslant \epsilon]$.
	 
	 From now on, we fix a constant step size $\alpha \in (0,\bar{\alpha}]$.  Consider a sequence $(x_k)_{k \in \mathbb{N}} \in\mathcal{M}(f,\alpha,X_0,0) \subset \mathcal{M}(f,\alpha,[f\leqslant \Delta/2],0)$ along with an associated subgradient trajectory $x:[0,T]\rightarrow \mathbb{R}^n$ of $f$ up to the multiplicative constant $c$ for which $x(0)\in [f\leqslant \Delta/2] \subset [f\leqslant \Delta - \epsilon' L]$ and
	 $\|x_k - x(\alpha (k-\bar{k}))\|\leqslant \epsilon'$ for $k = \bar{k}, \ldots, \bar{k}+K$ where $\bar{k} = 0$ and $K:=\lfloor T/\alpha\rfloor$. By Lemmas \ref{lemma:containment} and \ref{lemma:decrease}, for $k=0,\hdots,K$, we have $f(x_k) \leqslant f(x(k\alpha)) + \epsilon' L \leqslant f(x(0)) + \epsilon'L \leqslant \Delta/2 + \epsilon' L \leqslant \Delta$ and
	 \begin{equation}
  \label{eq:decrease}
		f(x_k) \leqslant f(x_0)- c\int_0^{k\alpha} d(0,\partial f (x(s)))^2~ds + 2\epsilon' L. 
     \end{equation}
If $c\int_0^{K \alpha} d(0,\partial f (x(s)))^2~ds  \geqslant 3\epsilon' L$, then we have $f(x_{K}) \leqslant f(x_0) -3\epsilon' L +2\epsilon' L \leqslant \Delta/2$ so that we may apply Lemmas \ref{lemma:containment} and \ref{lemma:decrease} again with $\bar{k} = K$. Since the continuous function $f$ is bounded below on the compact set $[f\leqslant \Delta/2]$, this process with constant decrease can only be repeated finitely many times. Thus there exist $v\in \mathbb{N}$ and an absolutely continuous function (again denoted $x(\cdot)$) such that $f(x_k) \leqslant f(x_{vK})- c\int_0^{(k-vK)\alpha} d(0,\partial f (x(s)))^2~ds + 2\epsilon' L$ and $\|x_{k}-x(\alpha(k-vK))\| \leqslant \epsilon'$ for $k = vK, \ldots, (v+1)K$ where $c\int_0^{K \alpha} d(0,\partial f (x(s)))^2~ds  < 3\epsilon' L$. Hence there exists $t'\in [0,K \alpha]$ such that $d(0,\partial f(x(t')))^2 \leqslant 3\epsilon' L/(cK\alpha) \leqslant  3\epsilon' L/(cT/2) \leqslant M^2/4$, where we use the fact that $\epsilon' \leqslant cM^2T/(24L)$. Since $d(0,\partial f(x(t'))) \leqslant M/2$ and $f(x(t')) \leqslant \Delta$, by definition of $M$ in \eqref{eq:M} there exists $i\in \{1,\ldots,p\}$ such that $|f(x(t')) - f_i| < \epsilon/2$. We also have that $f(x(t')) \leqslant f(x(0)) \leqslant f(x_{vK}) + \epsilon'L \leqslant \Delta/2 + \epsilon'L$. Thus $f_i<f(x(t')) + \epsilon/2 \leqslant \Delta/2 + \epsilon'L + \epsilon/2 \leqslant \Delta/2+3\epsilon/8$. For $k' = vK,\ldots, (v+1)K$, we have
\begin{subequations}
\label{eq:fvc}
	\begin{align}
	|f(x_{k'}) - f_i| \leqslant & |f(x_{k'}) - f(x(\alpha(k' - vK)))| + |f(x(\alpha(k' - vK))) - f(x(t'))| + \label{eq:fvc1}\\[2mm]
	& |f(x(t')) - f_i| \label{eq:fvc2} \\[2mm]
		 \leqslant & L\|x_{k'} - x(\alpha(k' - vK))\| + |f(x(0)) - f(x(K \alpha))| + \epsilon/4\label{eq:fvc3}\\[2mm]
		 \leqslant & L\epsilon' + 3\epsilon'L + \epsilon/2\label{eq:fvc4}\\[2mm]
		 \leqslant & \epsilon/8+3\epsilon/8+\epsilon/2\label{eq:fvc5}\\[2mm]
		 = & \epsilon.
	\end{align}
\end{subequations}
Indeed, \eqref{eq:fvc1} is due to the triangular inequality. We invoke the Lipschitz constant $L$ of $f$ on $[f\leqslant \Delta]$ in order to bound the first term in \eqref{eq:fvc1}. In order to bound the second term in \eqref{eq:fvc1}, we use the fact that the composition $f\circ x$ is decreasing and $0 \leqslant \alpha(k' - vK) \leqslant t' \leqslant K\alpha$. \eqref{eq:fvc4} holds because $\|x_{k'} - x(\alpha(k' - vK))\| \leqslant \epsilon'$ and $|f(x(0)) - f(x(K \alpha))| = c\int_0^{K \alpha} d(0,\partial f (x(s)))^2~ds  < 3\epsilon' L$. \eqref{eq:fvc5} is due to $\epsilon' \leqslant \epsilon/(8L)$.

We next show that $f(x_k) \leqslant f_i + \epsilon$ for all $k \geqslant k':=vK$. Without loss of generality, we assume that $k'=0$ so that by \eqref{eq:fvc} we have $f(x_k) \leqslant f_i + \epsilon$ for $k = 0, \hdots, K$. We prove that $f(x_{K+1})\leqslant f_i+\epsilon$, hence $f(x_{k})\leqslant f_i + \epsilon$ for all $k\geqslant k'$ by induction. We distinguish two cases. If $f(x_1)<f_i - \epsilon$, then $f(x_{K+1}) \leqslant f(x_1) + 2\epsilon'L < f_i - \epsilon + \epsilon/4 \leqslant f_i + \epsilon$, where the first inequality follows from $x_1 \in [f\leqslant f_i-\epsilon] \subset [f\leqslant \Delta/2 + 3\epsilon/8 - \epsilon] \subset [f\leqslant \Delta/2]$ and Lemmas \ref{lemma:containment} and \ref{lemma:decrease}. If $x_1 \in [|f-f_i|\leqslant \epsilon]$, then let $x:[0,T]\rightarrow \mathbb{R}^n$ be an associated subgradient trajectory of $f$ up to the multiplicative constant $c$ such that $\|x_k - x(\alpha (k-1))\|\leqslant \epsilon'$ for $k = 1,\ldots, K+1$ and $x(0) \in [|f-f_i|\leqslant \epsilon]$. Note that for any $t \in [0,K\alpha]$, $f(x(K\alpha)) \leqslant f(x(t))\leqslant f(x(0)) \leqslant f_i+\epsilon \leqslant \Delta - \epsilon \leqslant \Delta - \epsilon' L$. By Lemmas \ref{lemma:containment} and \ref{lemma:decrease}, $x_{K+1} \in [f\leqslant\Delta]$ and $f(x_{K+1}) \leqslant f(x(K\alpha)) + \epsilon'L$. If $f(x(K\alpha)) \leqslant f_i + \epsilon/2$, we have that $f(x_{K+1}) \leqslant f(x(K\alpha)) + \epsilon'L <f_i+\epsilon/2+\epsilon/8 \leqslant f_i+\epsilon$, as desired. Otherwise, we have $f(x(t)) \in [f_i+\epsilon/2,f_i+\epsilon]$ for all $t \in [0,K \alpha]$. By the Kurdyka-\L{}ojasiewicz inequality, we have $d(0,\partial f(x(t))) \geqslant 1/\psi'(f(x(t)) - f_i) \geqslant 1/\psi'(\epsilon/2)>0$. According to \cite[Lemma 5.2, Theorem 5.8]{davis2020stochastic} (see also \cite{drusvyatskiy2015curves}), it holds that
\begin{subequations}
	\begin{align}
		f(x(K \alpha)) - f_i&\leqslant f(x(0)) - f_i- c\int_{0}^{K \alpha} d(0,\partial f (x(s)))^2~ds\\
		&\leqslant f(x(0)) - f_i- cK\alpha/\psi'(\epsilon/2)^2\\[1mm]
		&\leqslant f(x(0)) - f_i - cT/(2\psi'(\epsilon/2)^2)\\[2mm]
		&\leqslant \epsilon -  cT/(2\psi'(\epsilon/2)^2).
		\end{align}
\end{subequations}
Thus $f(x_{K+1}) - f_i \leqslant f(x(K \alpha)) - f_i + f(x_{K+1}) - f(x(K \alpha)) \leqslant \epsilon - cT/(2\psi'(\epsilon/2)^2) + \epsilon'L \leqslant \epsilon$, where we used the fact that $\epsilon' \leqslant (cT)/(2L\psi'(\epsilon/2)^2)$.

If $|f(x_k) - f_i| \leqslant \epsilon$ for all $k \geqslant k'$, then the conclusion of the theorem follows. Otherwise, there exists $\hat{k}\geqslant k'$ such that $f(x_{\hat{k}})< f_i - \epsilon \leqslant \Delta/2 + 3\epsilon/8 - \epsilon \leqslant \Delta/2$. Following the same argument as in the paragraph below \eqref{eq:decrease}, there exists $v'\in \mathbb{N}$ and an absolutely continuous function (again denoted $x(\cdot)$) such that $f(x_k) \leqslant f(x_{v'K})- c\int_0^{(k - \hat{k})\alpha} d(0,\partial f (x(s)))^2~ds + 2\epsilon' L$ and $\|x_{k}-x(\alpha(k-v'K))\| \leqslant \epsilon'$ for $k = \hat{k}+v'K, \ldots, \hat{k}+(v'+1)K$ where $c\int_0^{K \alpha} d(0,\partial f (x(s)))^2~ds  < 3\epsilon' L$. As before, it follows that there exist $t''\in [0,T]$ and $j \in \{1,2, \ldots, p\}$ such that $|f(x(t'')) - f_j| \leqslant \epsilon/2$. Since $f(x(t'')) \leqslant f(x(0)) \leqslant f(x_{\hat{k} + v'K}) + \epsilon'L\leqslant f(x_{\hat{k}}) + 3\epsilon'L < f_i - \epsilon+ 3\epsilon/8 = f_i - 5\epsilon/8$, it holds that $f_j<f_i$. Replicating \eqref{eq:fvc1}-\eqref{eq:fvc5}, we get $|f(x_{k''}) - f_j| \leqslant \epsilon$ for $k''= \hat{k} + v'K, \ldots, \hat{k} + (v'+1)K$. By the same argument as in the previous paragraph, we have $f(x_k) \leqslant f_j+ \epsilon$ for all $k \geqslant k'':=\hat{k} + (v'+1)K$. Since $f$ only has finitely many critical values, the conclusion of the theorem follows. 
\end{proof}

Theorem \ref{thm:converge} gives a ``weak convergence'' result, in the sense that the function values evaluated at the iterates eventually stabilize around some critical value. In fact, a ``strong convergence'' result regarding the distance between the iterates and the set of critical points can be obtained without any additional assumptions. This is the subject of the following corollary. Note that while Corollary \ref{cor:converge_critical} implies Theorem \ref{thm:converge}, it is not clear how to prove Corollary \ref{cor:converge_critical} without Theorem \ref{thm:converge}.
\begin{corollary}[Stability of iterates]
\label{cor:converge_critical}
Let $f:\mathbb{R}^n \rightarrow \mathbb{R}$ be a locally Lipschitz coercive tame function and let $\mathcal{M}$ be an iterative method with constant step size that is approximated by subgradient trajectories of $f$. For any bounded set $X_0 \subset \mathbb{R}^n$ and $\epsilon>0$, there exists $\bar{\alpha}>0$ such that for all $(x_k)_{k \in \mathbb{N}} \in \mathcal{M}(f,(0,\bar{\alpha}],X_0,0)$, there exist a connected component $C$ of the set of critical points of $f$ and $k_0 \in \mathbb{N}$ such that $d(x_k,C) \leqslant \epsilon$ for all $k \geqslant k_0$.
\end{corollary}
\begin{proof}
Let $f:\mathbb{R}^n \rightarrow \mathbb{R}$ be a locally Lipschitz coercive tame function and let $\mathcal{M}$ be an iterative method with constant step size that is approximated by subgradient trajectories of $f$. Let $X_0$ be a bounded subset of $\mathbb{R}^n$ and let $\epsilon>0$. By Theorem \ref{thm:converge}, there exists $\alpha_1,\Delta>0$ such that for all $(x_k)_{k\in\mathbb{N}} \in \mathcal{M}(f,(0,\alpha_1], X_0, 0)$, $f(x_k) \leqslant \Delta$ for all $k \in \mathbb{N}$.

Let $L$ denote a Lipschitz constant of $f$ on the compact set $[f\leqslant \Delta]$ and consider the quantity
\begin{equation}
    M := \inf\{d(0,\partial f(x)): d(x,S)\geqslant \epsilon/2, f(x)\leqslant \Delta\}>0,
\end{equation}
where $S$ is the set of critical points of $f$. Since $\mathcal{M}$ is approximated by subgradient trajectories of $f$, by Definition \ref{def:approx_flow_new} there exist $c,T>0$, and $\alpha_2 \in (0,\alpha_1]$ such that for all $\alpha \in (0,\alpha_2],\bar{k} \in \mathbb{N}$ and $(x_k)_{k\in \mathbb{N}} \in \mathcal{M}(f,\alpha,[f\leqslant \Delta], \bar{k})$ for which $x_0,\ldots,x_{\bar{k}} \in [f \leqslant \Delta]$, there exists a subgradient trajectory $x:[0,T]\rightarrow \mathbb{R}^n$ of $f$ up to the multiplicative constant $c$ for which $x(0)\in [f\leqslant\Delta]$ and $\|x_{k} - x((k - \bar{k})\alpha)\|\leqslant \epsilon'$ for $k = \bar{k}, \ldots, \bar{k} + \lfloor T/\alpha\rfloor$ where $\epsilon' := \min\{\epsilon/4,cM^2T/(16(1+L)),\epsilon^2/(32(1+L)cT)\}$. Again by Theorem \ref{thm:converge} there exists $\alpha_3 \in (0,\alpha_2]$ such that for all $(x_k)_{k \in \mathbb{N}} \in \mathcal{M}(f,(0,\alpha_3],X_0,0)$, there exist a critical value $f^*$ of $f$ and $k_0 \in \mathbb{N}$ such that $ |f(x_k) - f^*| \leqslant \epsilon'$ for all $k \geqslant k_0$. Let $\bar{\alpha} := \min\{\alpha_3, \epsilon'/(2c(1+L)),T/2\}$. 

Let $\alpha \in (0,\bar{\alpha}]$, $(x_k)_{k \in \mathbb{N}} \in \mathcal{M}(f,\alpha, X_0, 0)$, and fix a corresponding $f^*$ and $k_0$. We fix some $k \geqslant k_0$ from now on and show that $d(x_k,S)\leqslant \epsilon$. Since $(x_{k'})_{k'\in \mathbb{N}}\in \mathcal{M}(f,\alpha, [f\leqslant \Delta],k)$ and $(x_{k'})_{k'\in \mathbb{N}} \subset [f\leqslant \Delta]$,  there exists a subgradient trajectory $x:[0,T]\rightarrow \mathbb{R}^n$ of $f$ up to the multiplicative constant $c$ for which $x(0)\in [f\leqslant\Delta]$ and $\|x_{k'} - x(\alpha (k'-k))\|\leqslant \epsilon'$ for $k' = k, \ldots, k + K$ where $K:= \lfloor T/\alpha\rfloor$. By Lemma \ref{lemma:decrease}, we have
 \begin{equation}
 \label{eq:infgrad}
   c\int_0^{K \alpha} d(0,\partial f (x(s)))^2~ds \leqslant f(x_k) - f(x_{k + K}) + 2\epsilon' L \leqslant 2\epsilon' (1+L).
 \end{equation}
Thus there exists $t' \in [0,K \alpha]$ such that $d(0,\partial f(x(t')))^2\leqslant 2\epsilon' (1+L)/(cK\alpha) \leqslant 2\epsilon' (1+L)/(cT/2) \leqslant M^2/4$, where we use the fact that $\epsilon' \leqslant cM^2T/(16(1+L))$. As $f(x(t')) \leqslant f(x(0))\leqslant \Delta$, we have $d(x(t'),S) \leqslant \epsilon/2$. It now suffices to show that $\|x_k-x(t')\| \leqslant \epsilon/2$. Notice that $\|x_k - x(0)\| \leqslant \epsilon' \leqslant \epsilon/4$ and
\begin{subequations}
    \begin{align}
        \|x(0) - x(t')\| &\leqslant \int_0^{t'} \|x'(s)\|~ds \label{eq:fvl-a}\\
        &= \int_0^{t'} c~d(0,\partial f(x(s)))~ds\label{eq:fvl-b}\\
        &\leqslant \sqrt{\int_0^{t'} c~ds} \sqrt{\int_0^{t'} c~d(0,\partial f(x(s)))^2~ds}\label{eq:fvl-c}\\
        &\leqslant \sqrt{\int_0^{T} c~ds} \sqrt{\int_0^{K \alpha}
        c~d(0,\partial f(x(s)))^2~ds}\label{eq:fvl-d}\\[2mm]
        &\leqslant \sqrt{cT} \sqrt{2\epsilon' (1+L)}\label{eq:fvl-e}\\[2mm]
        &\leqslant \epsilon/4.\label{eq:fvl-f}
    \end{align}
\end{subequations}
Indeed, \eqref{eq:fvl-a} is due to triangular inequality. \eqref{eq:fvl-b} is a consequence of \eqref{eq:ivp} and \cite[Lemma 5.2, Theorem 5.8]{davis2020stochastic} (see also \cite{drusvyatskiy2015curves}). \eqref{eq:fvl-c} is due to the Cauchy-Schwarz inequality. \eqref{eq:fvl-f} is due $\epsilon' \leqslant \epsilon^2/(32(1+L)cT)$. Summing up, we have $|d(x_k,S) - d(x(t'),S)| \leqslant \|x_k-x(t')\|\leqslant \|x_k-x(0)\|+\|x(0)-x(t')\| \leqslant \epsilon/2$ and thus $d(x_k,S) \leqslant \epsilon$. 

We have just shown that for all $d(x_k,S) \leqslant \epsilon$ for all $k\geqslant k_0$. Since $x_k \in [f \leqslant \Delta]$, we have that $d(x_k, S') = d(x_k,S) \leqslant \epsilon$, where $S':= S\cap B([f\leqslant \Delta],2\epsilon)$. By the cell decomposition theorem \cite[(2.11) p. 52]{van1998tame}, the definable compact set $S'$ has finitely many compact connected components $C_1,\hdots,C_q$. Thus for each $k\geqslant k_0$, there exists $i_k \in \{1,\hdots,q\}$ such that $d(x_k,C_{i_k}) \leqslant \epsilon$. We next show that $d(x_{k+1},C_{i_k})\leqslant \epsilon$, so that $i_k$ can actually be chosen independently of $k$. Naturally, we have $d(C_i,C_j):=\inf \{ \|x-y\| : (x,y) \in C_i \times C_j\}>0$ for all $i\neq j$, otherwise $C_i \cap C_j \neq \emptyset$. Without loss of generality, we may assume that $\epsilon \leqslant \min \{ d(C_i,C_j) : i \neq j\}/4$. It follows that, for all $j\neq i_k$, we have $d(x_k,C_j) \geqslant d(C_{i_k},C_j) - d(x_k,C_{i_k}) \geqslant 4\epsilon - \epsilon = 3\epsilon$. Similar to \eqref{eq:fvl-a}-\eqref{eq:fvl-f}, we have $\|x(0)-x(\alpha)\| \leqslant \sqrt{c\alpha}\sqrt{2\epsilon'(1+L)} \leqslant \epsilon'$ since $\alpha \leqslant \bar{\alpha} \leqslant \epsilon'/(2c(1+L))$. Thus $\|x_{k+1}-x_k\| \leqslant \|x_{k+1}-x(\alpha)\| + \|x(\alpha)-x(0)\| + \|x(0)-x_k\| \leqslant 3\epsilon' \leqslant \epsilon$. Hence $d(x_{k+1},C_j) \geqslant d(x_k,C_j) - \|x_k - x_{k+1}\| \geqslant 3\epsilon - \epsilon = 2\epsilon$ for all $j\neq i_k$. Since $d(x_{k+1},S) = \min\{d(x_{k+1},C_j): j=1, \hdots, q\} \leqslant \epsilon$, we conclude that $d(x_{k+1},C_{i_k}) \leqslant \epsilon$. 
\end{proof}

\begin{remark}
\label{remark:coercive}
The assumption that $f$ is coercive in Theorem \ref{cor:converge_critical} and Corollary \ref{cor:converge_critical} can be replaced by requiring the iterates to be uniformly bounded for all sufficiently small step sizes when initialized in $X_0$. In other words, we can ask for there to exist $\bar{\alpha},r>0$ such that $\mathcal{M}(f,(0,\bar{\alpha}],X_0,0) \subset B(0,r)^{\mathbb{N}}$. Indeed, one can then apply our results to a coercive function $f_r$ which coincides with $f$ in $B(0,2r)$, namely $f_r(x):=  f(P_{B(0,2r)}(x)) + d(x,B(0,2r))$ for all $x \in \mathbb{R}^n$ where $P_{B(0,2r)}$ is the projection on $B(0,2r)$. It is clear that $f_r$ is definable and coercive. In order to show that $f_r$ is Lipschitz continuous, it suffices to prove $g_r(x):= f(P_{B(0,2r)}(x))$ is locally Lipschitz. Let $L>0$ denote a Lipschitz constant of $f$ in $B(0,2r)$. For all $x,y\in \mathbb{R}^n$, we have $\|g_r(x) - g_r(y)\| = \|f(P_{B(0,2r)}(x)) - f(P_{B(0,2r)}(y))\| \leqslant L\|P_{B(0,2r)}(x) - P_{B(0,2r)}(y)\| \leqslant L\|x - y\|$. 
\end{remark}
\section{Approximation of first-order methods by subgradient trajectories}
\label{sec:Continuous-time}
The theory we developed in the previous section provides a unified framework under which global stability of iterative methods with constant step sizes can be established. In this section, we show that all of the first-order methods that we mentioned in Section \ref{sec:Introduction} are approximated by subgradient trajectories under appropriate assumptions on the objective functions. As a result, Theorem \ref{thm:converge} and Corollary \ref{cor:converge_critical} can be applied to conclude global stability of those methods. We need the following lemma in order to prove the approximation of random reshuffling with momentum.
\begin{lemma}
\label{lemma:O(alpha)}
Let $f_1,\ldots,f_N$ be locally Lipschitz,  $X \subset \mathbb{R}^n$ be bounded, $\delta \geqslant 0$, $\beta\in(-1,1)$, and $\gamma\in\mathbb{R}$. There exist $\delta', \bar{\alpha} > 0$ such that for all $\alpha\in (0,\bar{\alpha}]$, $\bar{k} \in \mathbb{N}$, and sequence $(x_{k,i})_{(k,i)\in \mathbb{N}\times \{0,\ldots, N\}}$ generated by random reshuffling with momentum (Algorithm \ref{alg:rrm}) for which $x_{0},\hdots,x_{\bar{k}}\in X$, we have 
\begin{subequations}
    \begin{align*}
        \|x_{k,i}-x_{k,i-1}\| \leqslant \delta'\alpha
    \end{align*}
\end{subequations}
for $k = 0,\ldots, \bar{k}$ and $i = 0,\ldots, N$.
\end{lemma}
\begin{proof}
    Let $r>0$ such that $x_{-1,0} \in B(0,r/2)$ and $X \subset B(0,r/2)$. Since $f_1,f_2, \ldots, f_N$ are locally Lipschitz, their corresponding Clarke subdifferentials $\partial f_1, \partial f_2, \ldots, \partial f_N$ are upper semicontinuous \cite[2.1.5 Proposition (d)]{clarke1990} with compact values \cite[2.1.2 Proposition (a)]{clarke1990}. Thus, by \cite[Proposition 3 p. 42]{aubin1984differential} there exists $r'>\delta$ such that $\cup_{i = 1}^N\partial f_i(B(0,r)) \subset B(0,r')$. Let
    \begin{equation*}
        \delta':= \frac{r'}{1 - |\beta|}~~~\text{and}~~~\bar{\alpha}:= \frac{r}{2\delta' ( N+ |\gamma|)}.
    \end{equation*}
    Fix any  $\alpha\in (0,\bar{\alpha}]$, $\bar{k} \in \mathbb{N}$, and sequence $(x_{k,i})_{(k,i)\in \mathbb{N}\times \{0,\ldots, N\}}$ generated by random reshuffling with momentum (Algorithm \ref{alg:rrm}) for which $x_{0},\hdots,x_{\bar{k}}\in X$. We will prove the lemma using induction on $(k,i)$ with the total order $\preccurlyeq$ defined by $(k_1,i_1)\preccurlyeq (k_2,i_2)$ if $k_1<k_2$ or $k_1 = k_2$ and $i_1\leqslant i_2$. For the base case, note that $\|x_{0,0} - x_{0,-1}\| \leqslant \delta \alpha <r'\alpha \leqslant \delta' \alpha$. Now fix any $(k,i) \in \{0,\ldots \bar{k}\}\times \{0,\ldots,N\}$ and assume that $\|x_{k',i'}-x_{k',i'-1}\| \leqslant \delta'\alpha$ for all $(k',i')\preccurlyeq (k,i-1)$ (we identify $(k'-1,N-1)$ with $(k',-1)$ for notational simplicity; possibly negative indices $i$ are treated similarly throughout the paper). Then
    \begin{subequations}
        \begin{align}
            \|y_{k,i}\| &\leqslant \|y_{k,i} - x_{k,i-1}\| + \|x_{k,i-1} - x_{k,0}\| + \|x_{k,0}\| \label{eq:ybound-a}\\
            &\leqslant |\gamma| \|x_{k,i-1} - x_{k,i-2}\| + \sum_{j = 1}^{i-1}\|x_{k,j} - x_{k,j-1}\| + \|x_{k}\|\label{eq:ybound-b}\\
            &\leqslant |\gamma|\delta'\alpha + (i-1)\delta'\alpha + \frac{r}{2}\label{eq:ybound-c}\\
            &\leqslant (|\gamma| + N -1)\delta'\alpha + \frac{r}{2}\label{eq:ybound-d}\\
            &\leqslant r \label{eq:ybound-e}.
        \end{align}
    \end{subequations}
Above, we use the triangular inequality in \eqref{eq:ybound-a}. We apply the update rule and again the triangular inequality to obtain \eqref{eq:ybound-b}. \eqref{eq:ybound-c} is a result of the inductive hypothesis. \eqref{eq:ybound-d} and \eqref{eq:ybound-e} follow from $i\leqslant N$ and $\alpha \leqslant \bar{\alpha}:= r/(2\delta'(|\gamma| + N - 1))$ respectively.

Thus $x_{k,i} - x_{k,i-1} - \beta (x_{k,i-1} - x_{k,i-2}) \in -\alpha \partial f_{\sigma_i^k}(y_{k,i}) \subset \alpha B(0,r')$. Therefore,
\begin{subequations}
    \begin{align*}
        \|x_{k,i} - x_{k,i-1}\| &\leqslant |\beta|\|x_{k,i-1} - x_{k,i-2}\| +  r'\alpha\\
        &\leqslant |\beta| \delta'\alpha + r'\alpha\\
        &= \delta'\alpha.
    \end{align*}
\end{subequations} 
\end{proof}

Recall that a locally Lipschitz function $f:\mathbb{R}^n\rightarrow\mathbb{R}$ is subdifferentially regular \cite[2.3.4 Definition]{clarke1990} if its generalized directional derivative agrees with the classical directional derivative, that is to say, we have 
\begin{equation*}
    \limsup_{\scriptsize\begin{array}{c} y\rightarrow x \\
    t \searrow 0
    \end{array}
    } \frac{f(y+th)-f(y)}{t} ~ = ~ \lim_{
    t \searrow 0
    } \frac{f(x+th)-f(x)}{t}
\end{equation*}
for all $x\in \mathbb{R}^n$ and $h \in \mathbb{R}^n$, and the limit on the right hand side exists. We assume subdifferential regularity in Propositon \ref{prop:rrm_flow} in order to guarantee that $\partial (f_1+\cdots+f_N) = \partial f_1 + \cdots + \partial f_N$, while in general we only know that $\partial (f_1+\cdots+f_N) \subset \partial f_1 + \cdots + \partial f_N$ holds \cite[2.3.3 Proposition]{clarke1990} (see Remark \ref{remark:need_reg}). If we do not assume subdifferential regularity, Proposition \ref{prop:rrm_flow} still holds with the same proof if in the conclusion we replace ``approximated by subgradient trajectories of $f$'' by ``approximated by trajectories of the conservative field $(\partial f_1+ \cdots +\partial f_N)/N$ \cite[Definition 1, Corollary 4]{bolte2020conservative} of $f$''. Theorem \ref{thm:converge} and Corollary \ref{cor:converge_critical} then hold with critical values and points associated with the conservative field under the additional assumption that $f_1,\hdots,f_N$ are definable \cite[Theorems 5 and 6]{bolte2020conservative}. On the other hand, since Algorithm \ref{alg:mag} does not consider the composite structure of the objective function, we do not require subdifferential regularity in order to obtain approximation and stability guarantees, as can be seen in Table \ref{tab:assumption}.

\begin{proposition}
\label{prop:rrm_flow}
Random reshuffling with momentum (Algorithm \ref{alg:rrm}) is approximated by subgradient trajectories of composite functions $f = (f_1 +\cdots + f_N)/N$ up to the multiplicative constant $N/(1-\beta)$ where $f_1,\ldots, f_N$ are locally Lipschitz and subdifferentially regular.
\end{proposition}
\begin{proof}
Let $\mathcal{M}$ denote the random reshuffling with fixed momentum parameters $\beta \in (-1,1)$, $\gamma \in \mathbb{R}$, and $\delta>0$. Let $X_0, X_1 \subset \mathbb{R}^n$ be compact sets and consider $r>0$ such that  $X_0, X_1 \subset B(0,r/2) \subset \mathbb{R}^n$. By Lemma \ref{lemma:O(alpha)}, there exist $\delta', \bar{\alpha} > 0$ such that for all $\alpha\in (0,\bar{\alpha}]$, $\tilde{k} \in \mathbb{N}$, and sequence $(x_{k,i})_{(k,i)\in \mathbb{N}\times \{0,\ldots, N\}}$ generated by random reshuffling with momentum (Algorithm \ref{alg:rrm}) for which $x_{0},\hdots,x_{\tilde{k}}\in B(0,r)$, we have that $\|x_{k,i}-x_{k,i-1}\| \leqslant \delta'\alpha$
for $k = 0,\ldots, \tilde{k}$ and $i = 0,\ldots, N$. Let $T:= r/(4\delta' N\max\{1, |\gamma|\})>0$ and $\bar{k}\in \mathbb{N}$.  We next show that any sequence $(x_{k,i})_{(k,i)\in \mathbb{N}\times \{0,\ldots, N\}}$ and $(y_{k,i})_{(k,i)\in \mathbb{N}\times \{1,\ldots, N\}}$ generated by Algorithm \ref{alg:rrm} with step size $\alpha\in (0,\min\{T/2,\bar{\alpha}\}]$ such that $x_0, \ldots, x_{\bar{k}} \in X_1$ and $x_{\bar{k}} \in X_0$ satisfy $x_{k,0},\ldots,x_{k,N}, y_{k,1}, \ldots, y_{k,N}\in B(0,r)$ for $k = \bar{k}, \ldots,\bar{k} + K-1$ where $K:= \lfloor T/\alpha\rfloor+1$.

Fix any such $\alpha$ and sequence generated by Algorithm \ref{alg:rrm}. Note that $\alpha K = \alpha (\lfloor T/\alpha\rfloor+1) \leqslant 2T$ and thus $\alpha \leqslant (2T)/K = r/(2K\delta' N\max\{1, |\gamma|\})$. As $x_0,\ldots,x_{\bar{k}_m} \in B(0,r)$, we have that $\|x_{\bar{k},i}\| \leqslant \|x_{\bar{k}}\| + \sum_{j = 1}^i\|x_{k,j} - x_{k,j-1}\| \leqslant r/2 +i\delta'\alpha \leqslant r/2 + ir/(2NK) \leqslant r/2 + r/(2K)$ for $i = 1,\ldots, N$, where we apply Lemma \ref{lemma:O(alpha)} with $\tilde{k} = \bar{k}$ in the second last inequality. In particular, $x_{\bar{k}+1} = x_{\bar{k},N} \in B(0,r/2 + r/(2K))$. Apply the previous argument recursively, we have that $x_{k,i} \in B(0,r/2+ (k-\bar{k})r/(2K) + ir/(2NK)) \subset B(0,r)$ for $k = \bar{k},\ldots, \bar{k} + K-1$. By the update rule of Algorithm \ref{alg:rrm}, $\|y_{k,i}\| \leqslant \|y_{k,i} - x_{k,i-1}\| + \|x_{k,i-1}\| = |\gamma|\|x_{k,i-1} - x_{k,i-2}\| + \|x_{k,i-1}\| \leqslant |\gamma|\delta'\alpha +r/2+ (k-\bar{k})r/(2K) + ir/(2NK) \leqslant  r/(2NK) +r/2+ (k-\bar{k})r/(2K) + ir/(2NK)\leqslant r$ for $k = \bar{k}, \ldots, \bar{k}+K-1$ and $i = 1, \ldots, N$.

Let $(\alpha_m)_{m\in \mathbb{N}}$ be a positive sequence that converges to zero and let $(\bar{k}_m)_{m\in \mathbb{N}}$ be a sequence of natural numbers. For each $m \in \mathbb{N}$, we attribute a sequence of iterates $(x^m_k)_{k\in \mathbb{N}} \in \mathcal{M}(f,\alpha_m, X_0,\bar{k}_m)$ such that $x_0, \ldots, x_{\bar{k}} \in X_1$. We may assume that $\alpha_m \in (0,\min\{T/2,\bar{\alpha}\}]$ for any $m$, then  $x^m_{k,0},\ldots,x^m_{k,N}, y^m_{k,1}, \ldots, y^m_{k,N}\in B(0,r)$ for $k = \bar{k}_m, \ldots, \bar{k}_m + \lfloor T/\alpha_m\rfloor$. Consider the linear interpolation of the iterates $x^m_{\bar{k}_m},x^m_{\bar{k}_m + 1}, \ldots, x^m_{\bar{k}_m + \lfloor T/\alpha_m\rfloor+1}$, that is to say, the function $\bar{x}^{m}(\cdot)$ defined from $[0,T]$ to $\mathbb{R}^n$ by 
\begin{equation*}
    \bar{x}^{m}(t) := x_k^{m} + (t-\alpha_m (k - \bar{k}_m))\frac{ x_{k+1}^{m}-x_k^{m}}{\alpha_m}    
\end{equation*}
for all $t\in [\alpha_m (k - \bar{k}_m) , \min\{\alpha_m(k - \bar{k}_m+1),T\}]$ and $k\in \{\bar{k}_m, \ldots, \bar{k}_m + \lfloor T/\alpha_m\rfloor\}$. Since $B(0,r)$ is convex, it holds that $\|\bar{x}^{m}(t)\| \leqslant r $ for all $t\in [0,T]$. We also have $\|(\bar{x}^{m})'(t)\| = \|(x_{k+1}^{m}-x_k^{m})/\alpha_m\| \leqslant \sum_{i = 1}^N \|x_{k,i}^m - x_{k,i-1}^m\|/\alpha_m \leqslant N \delta'$ for all $t\in [\alpha_m (k - \bar{k}_m) , \min\{\alpha_m(k - \bar{k}_m+1),T\}]$ and $k\in \{\bar{k}_m, \ldots, \bar{k}_m + \lfloor T/\alpha_m\rfloor\}$. By successively applying the Arzel\`a-Ascoli and the Banach-Alaoglu theorems (see \cite[Theorem 4 p. 13]{aubin1984differential}), there exist a subsequence (again denoted by $(\alpha_{m})_{m\in \mathbb{N}}$) and an absolutely continuous function $x:[0,T]\rightarrow \mathbb{R}^n$ such that $\bar{x}^{m}(\cdot)$ converges uniformly to $x(\cdot)$ and $(\bar{x}^{m})'(\cdot)$ converges weakly to $x'(\cdot)$ in $L^1([0,T],\mathbb{R}^n)$.  We next verify that the limit $x(\cdot)$ is a solution to the differential inclusion with initial condition
\begin{equation}
\label{eq:ivp_rrm}
	x'(t) \in -\frac{1}{1-\beta}\sum_{i = 1}^N \partial f_i(x(t)),~~~ \text{for almost every}~ 
                t \in [0,T], ~~~ x(0) \in X_0.
\end{equation}
By subdifferential regularity of $f_1, \ldots, f_N$, we have $\partial (\sum_{i = 1}^N f_i) = \sum_{i = 1}^N \partial f_i$ \cite[p. 40, Corollary 3]{clarke1990}. It is thus easy to see that such $x(\cdot)$ is a subgradient trajectory of $f$ up to the multiplicative constant $c:= N/(1-\beta)>0$.

For any fixed $m\in \mathbb{N}$, we have that
\begin{equation}
\label{eq:rrm_onestep}
    x_{k,i}^m - x_{k,i-1}^m - \beta (x_{k,i-1}^m - x_{k,i-2}^m) \in -\alpha_m \partial f_{\sigma_i^k} (y_{k,i}^m)
\end{equation}
for all $k\in \{\bar{k}_m, \ldots, \bar{k}_m + \lfloor T/\alpha_m\rfloor \rfloor\}$ and $i\in \{0,\ldots,N\}$. For any fixed $k$, summing \eqref{eq:rrm_onestep} up for $i = 1,\ldots, N$ yields
\begin{equation*}
    x_{k+1,0}^m - x_{k,0}^m - \beta (x_{k+1,-1}^m - x_{k,-1}^m) \in -\alpha_m \sum_{i = 1}^N\partial f_{\sigma_i^k} (y_{k,i}^m).
\end{equation*}

Consider the linear interpolation of the iterates $x^m_{\bar{k}_m,-1},x^m_{\bar{k}_m + 1,-1}, \ldots, $ \\$ x^m_{\bar{k}_m + \lfloor T/\alpha_m\rfloor+1,-1}$, that is to say, the function $\bar{x}_{-1}^{m}(\cdot)$ defined from $[0,T]$ to $\mathbb{R}^n$ by 
\begin{equation*}
    \bar{x}_{-1}^{m}(t) := x_{k,-1}^{m} + (t-\alpha_m (k - \bar{k}_m))\frac{ x_{k+1,-1}^{m}-x_{k,-1}^{m}}{\alpha_m}    
\end{equation*}
for all $t\in [\alpha_m (k - \bar{k}_m) , \min\{\alpha_m(k - \bar{k}_m+1),T\}]$ and $k\in \{\bar{k}_m, \ldots, \bar{k}_m + \lfloor T/\alpha_m\rfloor\}$.

For almost every $t\in (0,T)$ and any neighborhood $\mathcal{N}$ of $0$, there exists $m_0\in \mathbb{N}$ such that for any $m \geqslant m_0$, there exists $k \in \{\bar{k}_m, \ldots, \bar{k}_m + \lfloor T/\alpha_m\rfloor\}$ such that
\begin{subequations}
    \begin{align}
        (\bar{x}^m)'(t) - \beta (\bar{x}_{-1}^m)'(t) &= \frac{x_{k+1,0}^m - x_{k,0}^m}{\alpha_m} - \beta \frac{x_{k+1,-1}^m - x_{k,-1}^m}{\alpha_m}\label{eq:rrm_neighbor_a} \\
        &\in -\sum_{i = 1}^N\partial f_{\sigma_i^k} (y_{k,i}^m)\label{eq:rrm_neighbor_b}\\
        &\subset -\sum_{i = 1}^N \left( \partial f_{\sigma_i^k}(x(t))+\mathcal{N}/N\right) \label{eq:rrm_neighbor_c}\\
        &\subset -\sum_{i = 1}^N  \partial f_{i}(x(t)) + \mathcal{N},\label{eq:rrm_neighbor_d}
    \end{align}
\end{subequations}
where \eqref{eq:rrm_neighbor_c} follows from upper semi-continuity of $\partial f_i$ \cite[2.1.5 Proposition (d)]{clarke1990} and
\begin{subequations}
    \begin{align*}
        \|y_{k,i}^m - x(t)\|&\leqslant \left\|y_{k,i}^m - \bar{x}^m(t)\right\| + \left\|\bar{x}^m(t) - x(t) \right\|\\
        &= \left\|y_{k,i}^m - x_k^m - (t-\alpha_m (k - \bar{k}_m))\frac{x_{k+1}^m - x_k^m}{\alpha_m}\right\|+ \left\|\bar{x}^m(t) - x(t) \right\|\\
        &\leqslant \left\|y_{k,i}^m - x_k^m\right\| + \left\|x_{k+1}^m - x_k^m\right\| + \left\|\bar{x}^m(t) - x(t) \right\|\\
        &\leqslant (|\gamma|+i)\delta'\alpha_m + N\delta'\alpha_m + \left\|\bar{x}^m(t) - x(t) \right\|\\
        &\rightarrow 0
    \end{align*}
\end{subequations}
as $m\rightarrow \infty$. It remains to show that $(\bar{x}^m_{-1})'(\cdot)$ converges weakly to $x'(\cdot)$ in $L^1([0,T],\mathbb{R}^n)$. Indeed, by \cite[Convergence Theorem p. 60]{aubin1984differential}, it then holds that $(x,(1-\beta)x') \in \mathrm{graph}(\sum_{i = 1}^N \partial f_i)$ and thus \eqref{eq:ivp_rrm} follows.

For any $m\in \mathbb{N}$ and almost every $s \in [0,T]$, $\|(\bar{x}^m_{-1})'(s)\| = \|(x_{k+1,-1}^m - x_{k,-1}^m)/\alpha_m\| \leqslant \sum_{i = 0}^{N-1} \|x_{k,i}^m - x_{k,i-1}^m\|/\alpha_m \leqslant N \delta'$ for some $k\in \mathbb{N}$. Thus it suffices to show that for all $t \in [0,T]$,
\begin{equation*}
    \int_0^t (\bar{x}^m_{-1})'(s)~ds \rightarrow \int_0^t x'(s)~ds.
\end{equation*}
Indeed, $\left\| \int_0^t (\bar{x}^m_{-1})'(s)~ds -  \int_0^t (\bar{x}^m)'(s)~ds\right\| = \cdots$
\begin{subequations}
    \begin{align*}
         =& \left\|\bar{x}^m_{-1}(t) - \bar{x}^m_{-1}(0) - (\bar{x}^m(t) - \bar{x}^m(0))\right\|\\
        =& \left\|x_{k,-1}^{m} + (t-\alpha_m (k - \bar{k}_m))\frac{ x_{k+1,-1}^{m}-x_{k,-1}^{m}}{\alpha_m} - x_{\bar{k}_m,-1}^{m} \right. \\
        &  \left. - \left(x_{k,0}^{m} + (t-\alpha_m (k - \bar{k}_m))\frac{ x_{k+1,0}^{m}-x_{k,0}^{m}}{\alpha_m} - x_{\bar{k}_m,0}^{m}\right)\right\|\\
        \leqslant & \left\|x_{k+1,-1}^{m} - x_{k+1,0}^{m}\right\| + \left\|x_{k,-1}^{m} - x_{k,0}^{m}\right\| + \left\|x_{\bar{k}_m,-1}^{m} - x_{\bar{k}_m,0}^{m}\right\|\\
        \leqslant & ~ \delta'\alpha_m  + \delta'\alpha_m + \delta'\alpha_m 
        \rightarrow 0
    \end{align*}
\end{subequations}
where $k = \bar{k}_m + \lfloor t/\alpha_m\rfloor$. As $x'(\cdot)$ is a weak limit of $(\bar{x}^m)'(\cdot)$, $(\bar{x}^m_{-1})'(\cdot)$ converges weakly to $x'(\cdot)$.

To sum up, we have shown that for every sequence $(\alpha_m)_{m \in \mathbb{N}}$ of positive numbers converging to zero and every sequence $(\bar{k}_m)_{m\in \mathbb{N}}$ of natural numbers, there exists a subsequence of natural numbers for which the corresponding linear interpolations uniformly converge towards a solution of the differential inclusion \eqref{eq:ivp_rrm}. The conclusion of the proposition now easily follows. To see why, one can reason by contradiction and assume that there exists $\epsilon>0$ such that for all $\bar{\alpha}>0$, there exist $\hat{\alpha} \in (0,\bar{\alpha}]$, $\hat{k} \in \mathbb{N}$, and a sequence $(x_k^m)_{k \in \mathbb{N}}\in \mathcal{M}(f,\hat{\alpha}, X_0,\hat{k})$ such that $x_0^m,\ldots, x_{\hat{k}}^m \in X_1$, and for any solution $x(\cdot)$ to the differential inclusion \eqref{eq:ivp_rrm}, it holds that $\|x_k^m - x(\hat{\alpha} (k - \hat{k})) \| > \epsilon$ for some $k \in \{\hat{k},\hat{k}+1, \hdots, \hat{k}+\lfloor T/\hat{\alpha}\rfloor\}$. We can then generate a sequence $(\alpha_m)_{m \in \mathbb{N}}$ of positive numbers converging to zero and a sequence $(\bar{k}_m)_{m\in \mathbb{N}}$ of natural numbers such that, for any solution $x(\cdot)$ to the differential inclusion \eqref{eq:ivp_rrm}, it holds that $\|x_k^m - x(\alpha_m(k - \bar{k}_m)) \| > \epsilon$ for some $k \in \{\bar{k}_m,\bar{k}_m+1, \hdots, \bar{k}_m+\lfloor T/\alpha_m\rfloor\}$. Since there exists a subsequence $(\alpha_{\varphi(m)})_{m \in \mathbb{N}}$ such that $(\bar{x}^{\varphi(m)}(\cdot))_{m\in \mathbb{N}}$ uniformly converges to a solution to the differential inclusion \eqref{eq:ivp_rrm}, we obtain a contradiction. 
\end{proof}
\begin{remark}
\label{remark:need_reg}
    To further see why subdifferential regularity is required for applying Definition \ref{def:approx_flow_new} to Algorithm \ref{alg:rrm}, consider a locally Lipschitz coercive semi-algebraic function $f:= (f_1 + f_2 + f_3)/3$ where $f_1,f_2,f_3:\mathbb{R} \rightarrow \mathbb{R}$ are defined by
$$
    f_1(x) := \max\{x,0\}~~~,~~~f_2(x) := \min\{x,0\}~~~,~~~f_3(x) := x^2~~~,~~~\forall x\in \mathbb{R}.
$$
Notice that $f_1,f_2,f_3$ are all locally Lipschitz but $f_2$ is not subdifferentially regular. We have $0 \in \partial f_1(0) = [0,1], 0\in \partial f_2(0) = [0,1]$, and $0 \in \partial f_3(0) = \{0\}$. Thus random reshuffling with momentum can get stuck at $0$. Meanwhile $0 \not \in \partial f(0) = \{1/3\}$. Therefore, the conclusion of Corollary \ref{cor:converge_critical} does not apply to this example.
\end{remark}

\begin{proposition}
\label{prop:cd_flow}
Random-permutations cyclic coordinate descent method (Algorithm \ref{alg:rcd}) is approximated by subgradient trajectories of continuously differentiable functions $f:\mathbb{R}^n \rightarrow \mathbb{R}$.
\end{proposition}
\begin{proof}
Similar to the proof of Proposition \ref{prop:rrm_flow}, let $X_0 \subset \mathbb{R}^n$ be a compact subset. Consider $r>0$ such that $X_0 \subset B(0,r/2) \subset \mathbb{R}^n$. We would like to find $T>0$ such that for any $\alpha \in (0,T/2], \bar{k}\in \mathbb{N}$, any sequence $(x_{k,i})_{(k,i)\in\mathbb{N}\times \{0,\ldots,n\}}$ generated by Algorithm \ref{alg:rcd} with step size $\alpha$ for which $x_{\bar{k},0} \in X_0$, we have that $x_{k,i}\in B(0,r)$ for all $k = \bar{k}, \ldots, \bar{k} + K-1$ and $i = 0,1, \ldots, n$ where $K:= \lfloor T/\alpha\rfloor+1$. As $\nabla f$ is continuous, we have $M:= \sup\{\|\nabla f(x)\|: x\in B(0,r)\}<\infty$. Let $T:= r/(4Mn)>0$. Then $\alpha K = \alpha (\lfloor T/\alpha\rfloor+1) \leqslant 2T$ and $\alpha \leqslant 2T/K = r/(2KMn) $. For any $k = \bar{k},\ldots, \bar{k} + K-1$ and $i = 0, \ldots, n-1$, if one assumes that $x_{k,i} \in B(0,r)$, then $\|x_{k,i+1} - x_{k,i}\|= \|\alpha  \nabla_{\sigma^k_{i+1}} f(x_{k,i})\| \leqslant \alpha M \leqslant r/(2Kn)$. As $x_{\bar{k},0} \in B(0,r/2)$, by induction, we have $x_{k,i}\in B(0,r/2 + (kn+i)r/(2Kn)) \subset B(0,r)$ for any $k = \bar{k},\ldots, \bar{k} + K-1$ and $i = 0, \ldots, n$.

We next show that Algorithm \ref{alg:rcd} is approximated by subgradient trajectories, which are solutions to the following differential equation with initial condition
\begin{equation}
\label{eq:ivp_rcd}
	x'(t) = -\nabla f(x(t)),~~~ \text{for almost every}~ 
                t \in [0,T], ~~~ x(0) \in X_0.
\end{equation}
Denote by $\mathcal{M}$ the random-permutations cyclic coordinate descent method defined by Algorithm \ref{alg:rcd}. Let $(\alpha_m)_{m \in \mathbb{N}}$ denote a sequence of positive numbers that converges to zero and $(\bar{k}_m)_{m\in \mathbb{N}}$ be a sequence of natural numbers. Without loss of generality, we may assume that $(\alpha_m)_{m \in \mathbb{N}} \subset (0,T/2]$. For each $m \in \mathbb{N}$, we attribute a sequence of iterates $(x^m_k)_{k\in \mathbb{N}} \in \mathcal{M}(f,\alpha_m, X_0,\bar{k}_m)$. Consider the linear interpolation of the iterates $x^m_{\bar{k}_m},x^m_{\bar{k}_m + 1}, \ldots, x^m_{\bar{k}_m + \lfloor T/\alpha_m\rfloor+1}$, that is to say, the function $\bar{x}^{m}(\cdot)$ defined from $[0,T]$ to $\mathbb{R}^n$ by
\begin{equation*}
    \bar{x}^{m}(t) := x_k^{m} + (t-\alpha_m (k - \bar{k}_m))\frac{x_{k+1}^{m}-x_k^{m}}{\alpha_m}    
\end{equation*}
for all $t\in [\alpha_m k , \min\{\alpha_m(k+1),T\}]$ and $k\in \{\bar{k}_m,\hdots,\bar{k}_m + \lfloor T/\alpha_m \rfloor\}$. Recall that $x_k^m = x_{k,0}^m = x_{k-1,n}^m$. As we have shown in the first paragraph of the proof, $x_{k,i}\in B(0,r)$ for all $k = \bar{k}_m, \ldots, \bar{k}_m + \lfloor T/\alpha_m\rfloor$ and $i = 0,1, \ldots, n$, thus $x^m_{\bar{k}_m},x^m_{\bar{k}_m + 1}, \ldots, x^m_{\bar{k}_m + \lfloor T/\alpha_m\rfloor+1}\in B(0,r)$. Since $B(0,r)$ is convex, it holds that $\|\bar{x}^{m}(t)\| \leqslant r $ for all $t\in [0,T]$. Observe that $(\bar{x}^{m})'(t) = (x_{k+1}^{m}-x_k^{m})/\alpha_m = -\sum_{i = 1}^n \nabla_{\sigma^k_i} f(x_{k,i-1})$ for all $t\in (\alpha_m k , \min\{\alpha_m(k+1),T\})$ and $k\in \{\bar{k}_m,\hdots,\bar{k}_m+ \lfloor T/\alpha_m \rfloor\}$. Hence, we have that $$\|(\bar{x}^{m})'(t)\|  = \|\sum_{i = 1}^n \nabla_{\sigma^k_i} f(x_{k,i-1})\| \leqslant \sum_{i = 1}^n\|\nabla_{\sigma^k_i} f(x_{k,i-1})\|\leqslant nM$$ for almost every $t\in [0,T]$. By successively applying the Arzel\`a-Ascoli and the Banach-Alaoglu theorems (see \cite[Theorem 4 p. 13]{aubin1984differential}), there exist a subsequence (again denoted $(\alpha_{m})_{m\in \mathbb{N}}$) and an absolutely continuous function $x:[0,T]\rightarrow \mathbb{R}^n$ such that $\bar{x}^{m}(\cdot)$ converges uniformly to $x(\cdot)$ and $(\bar{x}^{m})'(\cdot)$ converges weakly to $x'(\cdot)$ in $L^1([0,T],\mathbb{R}^n)$.

For almost every $t\in [0,T]$ and any $\alpha_m$ in the sequence, $t\in (\alpha_{m} k , \min\{\alpha_{m}(k+1),T\})$ for some $k \in \{\bar{k}_m,\hdots,\bar{k}_m + \lfloor T/\alpha_m \rfloor\}$. We fix any such $t$ and any $\xi>0$ from now on. As $\nabla f$ is continuous, there exists $\delta>0$ such that $\|\nabla_i f(y) - \nabla_i f(x(t))\| \leqslant \|\nabla f(y) - \nabla f(x(t))\| \leqslant \xi/(2n)$, for all $y \in B(x(t),\delta)$ and $i = 1,\ldots, n$. Since $(\bar{x}^m(\cdot))_{m\in \mathbb{N}}$ converges uniformly to $x(\cdot)$, there exists $m_0\in \mathbb{N}$ such that $\|\bar{x}^m(t) - x(t)\| \leqslant  \min\{\epsilon/2,\delta/2\}$ for any $m \geqslant m_0$. As $\lim_{m \rightarrow \infty}\alpha_m = 0$, there exists $m_1\geqslant m_0$ such that $\alpha_m \leqslant \delta/(4nM)$ for all $m \geqslant m_1$. We next show that $\|(\bar{x}^{m})'(t) - \nabla f(x(t))\|\leqslant \xi/2$ for all $m\geqslant m_1$. Indeed, if it is the case, then
\begin{subequations}
    \begin{align*}
        (\bar{x}_m(t),\bar{x}'_m(t)) & \in B\left(x(t),\min\left\{\frac{\xi}{2},\frac{\delta}{2}\right\}\right) \times B\left(-\nabla f(x(t)),\frac{\xi}{2}\right) \\
        & \subset \mathrm{graph}\left(-\nabla f\right) + B(0,\xi)
    \end{align*}
\end{subequations}
and by \cite[Convergence Theorem p. 60]{aubin1984differential}, it holds that $x'(t)= -\nabla f(x(t))$ for almost every $t\in [0,T]$. The sequence of initial iterates $(x^{m}(0))_{m\in\mathbb{N}}$ lies in the compact set $X_0$, hence its limit $x(0)$ lies in $X_0$ as well. As a result, $x(\cdot)$ is a solution to the differential inclusion \eqref{eq:ivp_rcd}.

Note that for any $i = 1, \ldots, n$, 
\begin{subequations}
    \begin{align*}
        &\|x_{k,i-1}^m - \bar{x}^m(t)\|\\
        =& \left\|x_k^m - \alpha_m\sum\limits_{j = 1}^{i-1}\nabla_{\sigma^k_j} f(x_{k,j-1}^m) - x_k^{m} - (t-\alpha_m (k - \bar{k}_m))\frac{x_{k+1}^{m}-x_k^{m}}{\alpha_m}\right\|\\
        =& \left\|\alpha_m\sum\limits_{j = 1}^{i-1}\nabla_{\sigma^k_j} f(x_{k,j-1}^m) + (t-\alpha_m (k - \bar{k}_m))\frac{x_{k+1}^{m}-x_k^{m}}{\alpha_m}\right\|\\
        \leqslant& \alpha_m \left\|\sum\limits_{j = 1}^{i-1}\nabla_{\sigma^k_j} f(x_{k,j-1}^m)\right\| + \|x_{k+1}^{m}-x_k^{m}\|\\
        \leqslant& \alpha_m \sum\limits_{j = 1}^{i-1}\left\|\nabla_{\sigma^k_j} f(x_{k,j-1}^m)\right\| + \alpha_m \|(\bar{x}^m)'(t)\|\\
        \leqslant& \alpha_m (i-1)M + \alpha_m nM\\
        \leqslant& 2\alpha_m nM \leqslant \delta/2.
    \end{align*}
\end{subequations}
Thus, $\|x_{k,i-1}^m - x(t)\| \leqslant \|x_{k,i-1}^m - \bar{x}^m(t)\| + \|\bar{x}^m(t) - x(t)\| \leqslant \delta/2 + \min\{\xi/2,\delta/2\}\leqslant \delta$ for all $i = 1, \ldots , n$ and $m \geqslant m_1$. It follows that
\begin{subequations}
    \begin{align*}
        \left\|(\bar{x}^m)'(t) - \nabla f(x(t))\right\| &= \left\|\sum_{i = 1}^n\nabla_{\sigma^k_i} f(x_{k,i-1}^m) -  \nabla f(x(t))\right\|\\
        &= \left\|\sum_{i = 1}^n\left(\nabla_{\sigma^k_i} f(x_{k,i-1}^m) - \nabla_{\sigma^k_i} f(x(t))\right)\right\|\\
        &\leqslant \sum_{i = 1}^n \left\| \nabla_{\sigma^k_i} f(x_{k,i-1}^m) - \nabla_{\sigma^k_i} f(x(t))\right\|\\
        &\leqslant \sum_{i = 1}^n \frac{\xi}{2n} = \frac{\xi}{2}.
    \end{align*}
\end{subequations}

To sum up, we have shown that for every sequence $(\alpha_m)_{m \in \mathbb{N}}$ of positive numbers converging to zero and every sequence $(\bar{k}_m)_{m\in \mathbb{N}}$ of natural numbers, there exists a subsequence of natural numbers for which the corresponding linear interpolations uniformly converge towards a solution of the differential equation \eqref{eq:ivp_rcd}. The conclusion of the proposition now easily follows using the same argument as the last paragraph in the proof of Proposition \ref{prop:rrm_flow}. 
\end{proof}

Note that in the proof above, we do not make use of the set $X_1$ that appears in Definition \ref{def:approx_flow_new}. This is because that for any iterates $(x_k)_{k \in \mathbb{N}}$ generated by the random-permutations cyclic coordinate descent method, $(x_k)_{k\geqslant \bar{k}}$ is unrelated to $x_0,\ldots, x_{\bar{k} - 1}$ if given $x_{\bar{k}}$. Contrary to Proposition \ref{prop:rrm_flow}, Proposition \ref{prop:cd_flow} cannot be relaxed to locally Lipschitz functions. For example, the coordinate descent method can get stuck at $(1,1)$, which is not a critical point of $f(x_1,x_2) = \max\{|x_1|,|x_2|\}$.

We conclude this paper by illustrating Theorem \ref{thm:converge}, Corollary \ref{cor:converge_critical}, Proposition \ref{prop:rrm_flow}, and Proposition \ref{prop:cd_flow} on two examples. The first (Figure \ref{fig:nonsmooth}) is nonsmooth and the second (Figure \ref{fig:smooth}) is continuously differentiable. One can see that the iterates indeed track a subgradient trajectory up to a certain time, then go on to track another subgradient trajectory, after which they stabilize around a critical point.

\begin{figure}[ht!]
\centering
\begin{subfigure}{.49\textwidth}
  \centering
  \includegraphics[width=0.95\textwidth]{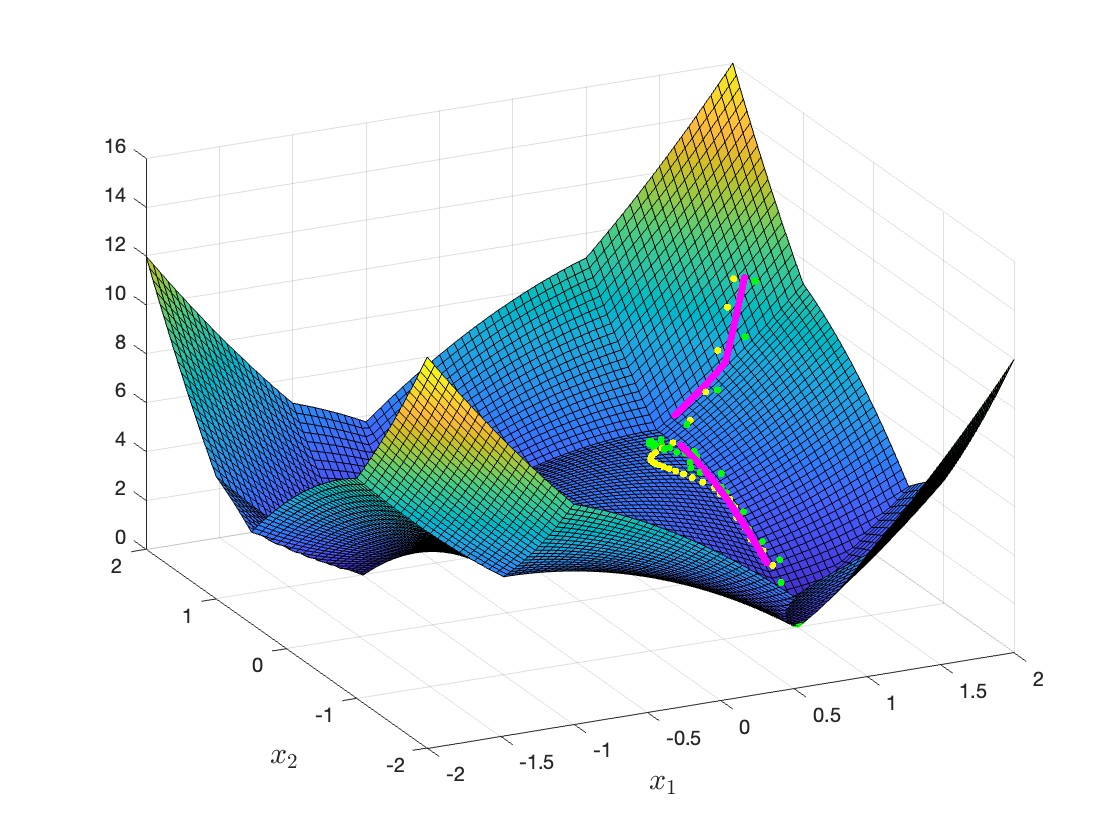}
  \caption{$f(x_1,x_2) = |x_1^2 - 1| + 2|x_1x_2+1| + |x_2^2 - 1|$.}
  \label{fig:nonsmooth}
\end{subfigure}
 \begin{subfigure}{.49\textwidth}
  \centering
   \includegraphics[width=0.95\textwidth]{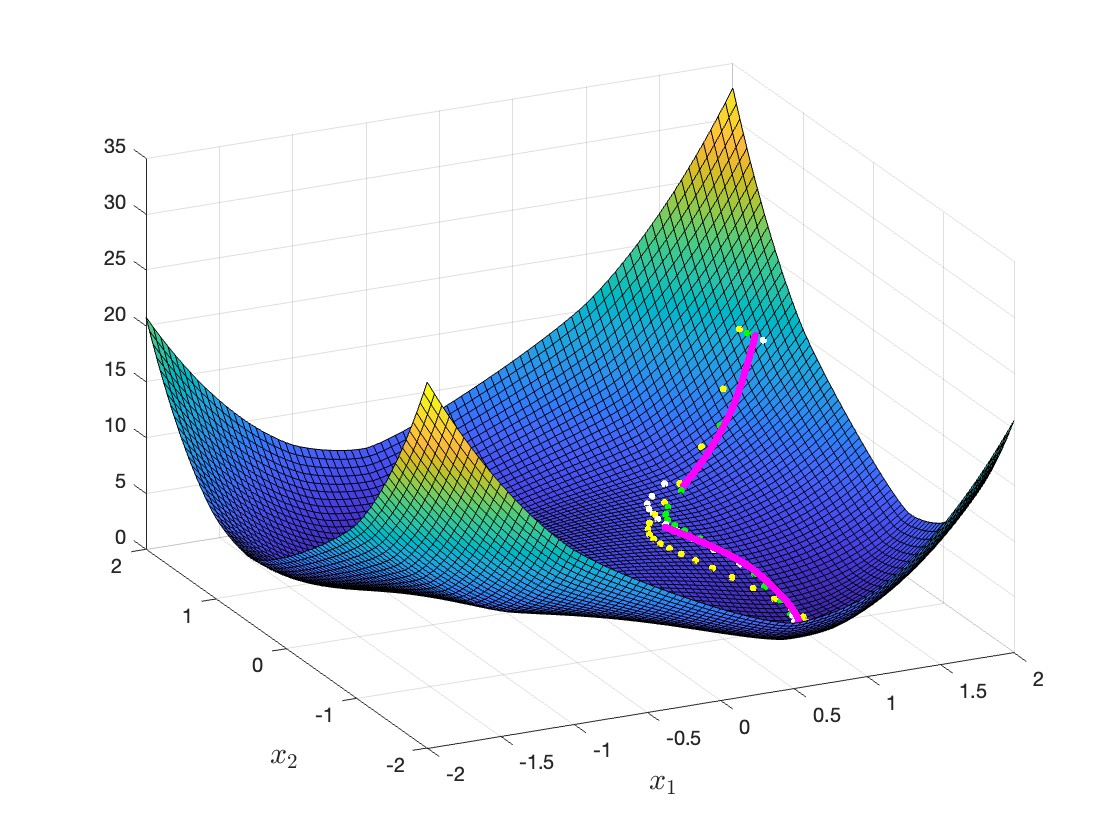}
  \caption{$f(x_1,x_2) = |x_1^2 - 1|^{3/2} + 2|x_1x_2+1|^{3/2} + |x_2^2 - 1|^{3/2}$.}
   \label{fig:smooth}
\end{subfigure}
\caption{The subgradient method with momentum, random reshuffling with momentum, and random-permutations cyclic coordinate descent method are in yellow, green, and white respectively. Subgradient trajectories are in magenta.}
\end{figure}
\section*{Acknowledgments}
We thank the reviewers and the co-editor for their valuable feedback.
\bibliographystyle{abbrv}    
\bibliography{mybib}

\end{document}